\documentclass[11pt]{amsart}
\usepackage{amsmath}
\usepackage{amsthm}
\usepackage{amsfonts}
\usepackage{amssymb}
\usepackage{lineno} 
\usepackage{tabularx}
\usepackage{enumerate}

\usepackage[left=3cm,right=3cm, top=3cm, bottom=3cm]{geometry}

\newcommand{\Z}{\mathbb{Z}}
\newcommand{\Q}{\mathbb{Q}}

\newcommand{\F}{\mathbb{F}}
\newcommand{\B}{\mathcal{B}}
\newcommand{\M}{\mathcal{M}}
\newcommand{\OO}{\mathcal{O}}
\newcommand{\RR}{\mathfrak{R}}
\newcommand{\Kbar}{\overline{K}}
\newcommand{\Tbar}{\overline{T}}
\newcommand{\psibar}{\overline{\psi}}
\newcommand{\ahat}{\widehat{a}}
\newcommand{\hhat}{\widehat{h}}
\newcommand{\chat}{\widehat{c}}
\newcommand{\fhat}{\widehat{f}}
\newcommand{\pihat}{\widehat{\pi}}
\newcommand{\omegahat}{\widehat{\omega}}
\newcommand{\alphahat}{\widehat{\alpha}}
\newcommand{\gammahat}{\widehat{\gamma}}
\newcommand{\atilde}{\widetilde{a}}
\newcommand{\ctilde}{\widetilde{c}}
\newcommand{\ftilde}{\widetilde{f}}
\newcommand{\pitilde}{\widetilde{\pi}}
\newcommand{\dst}{\displaystyle}

\newcommand{\Gal}{\operatorname{Gal}}

\theoremstyle{plain}

\newtheorem{theorem}{Theorem}[section]
\newtheorem{proposition}[theorem]{Proposition}
\newtheorem{corollary}[theorem]{Corollary}
\newtheorem{lemma}[theorem]{Lemma}

\theoremstyle{definition}
\newtheorem{definition}[theorem]{Definition}

\newtheorem{remark}[theorem]{Remark}

\title{Extensions of local fields given by $3$-term Eisenstein
polynomials}

\author[Fejzullahu]{Endrit Fejzullahu}
\address{Department of Mathematics \\
University of Florida \\
Gainesville, FL 32611 \\
USA}
\email{f.endrit@gmail.com}

\author[Keating]{Kevin Keating}
\address{Department of Mathematics \\
University of Florida \\
Gainesville, FL 32611 \\
USA}
\email{keating@ufl.edu}

\date{\today}

\begin{document}

\begin{abstract}
Let $K$ be a local field with residue characteristic
$p$ and let $L/K$ be a totally ramified extension of
degree $p^k$.  In this paper we show that if $L/K$ has
only two distinct indices of inseparability then there
exists a uniformizer $\pi_L$ for $L$ whose minimum
polynomial over $K$ has at most three terms.  This leads
to an explicit classification of extensions with two
indices of inseparability.  Our classification extends
work of Amano, who considered the case $k=1$.
\end{abstract}

\maketitle

\section{Introduction} \label{introduction}

Let $K$ be a local field with normalized discrete
valuation $v_K: K\to \mathbb{Z}\cup\{\infty\}$.  
Let $L/K$ be a totally ramified extension and let
$\pi_L$ be a uniformizer of $L$. Then the minimum
polynomial for $\pi_L$ over $K$ is Eisenstein, and
$L=K(\pi_L)$.
It
may be desirable to find a uniformizer for $L$ whose
minimum polynomial over $K$ has as few nonzero terms as
possible in order to give a succinct description of the
extension.  In \cite{amano},
Amano studied totally ramified extensions $L/K$ of
degree $p$, where $K$ is a finite extension of
$\mathbb{Q}_p$, and showed that there exists a
uniformizer $\pi_L$ for $L$ that is a root of an
Eisenstein polynomial over $K$ with at most three
nonzero terms. In this paper we study separable totally
ramified extensions $L/K$ of degree $p^k$ for which
there exists a uniformizer $\pi_L$ whose minimum
polynomial over $K$ has three nonzero terms.  In
particular, we show that if $L/K$ has exactly two
indices of inseparability then such a uniformizer always
exists.  We use this fact to give an explicit
classification of isomorphism classes of totally
ramified extensions $L/K$ of degree $p^k$ with two
indices of inseparability.

In Section~\ref{background} we define the indices of
inseparability.  In Section~\ref{quote} we discuss the
``stability theorem'' (Theorem~\ref{stability}) and the
``perturbation theorem'' (Theorem~\ref{perturbation}).
In Section~\ref{iteration} we state our main result,
Theorem~\ref{main}.  The proof of this theorem is found
in Sections~\ref{existence} and \ref{uniqueness}. In Section~\ref{splitting} we discuss the splitting fields of the polynomials obtained in Theorem~\ref{main}. 
In Section~\ref{examples} we provide some counterexamples that
rule out various possible generalizations of the results
of Section~\ref{iteration}.

\subsection*{Notation} Let $K$ be a field which is
complete with respect to a normalized discrete valuation
$v_K$.  Let
\begin{align*}
\OO_K&=\{\alpha\in K:v_K(\alpha)\geqslant 0\} \\
\M_K&=\{\alpha\in K:v_K(\alpha)\geqslant 1\}
\end{align*}
be the ring of integers and maximal ideal of $K$.  We
assume that the residue field $\Kbar=\OO_K/\M_K$ of $K$
is a perfect field of characteristic $p$.  Let $\RR$
denote the set of Teichm\"uller representatives of $K$.
Let $\pi_K$ be a fixed uniformizer for $K$; thus
$v_K(\pi_K)=1$ and $\M_K=\pi_K\OO_K$.  Let $\Omega$ be a
separable closure of $K$.  For a finite subextension
$L/K$ of $\Omega/K$ we get $v_L$, $\OO_L$, and $\M_L$ as
above, with $v_L$ normalized for $L$.  However, a
uniformizer $\pi_L$ for $L$ should not be viewed as
fixed.

Most of this work appeared in the first author's 2020
University of Florida Ph.D. thesis \cite{thesis}.

\section{Indices of inseparability}
\label{background}

Let $L/K$ be a finite separable totally ramified
extension.  The indices of inseparability of $L/K$ were
defined by Fried \cite{fried} in the case char$(K)=p$,
and by Heiermann \cite{heiermann} in general.  In this
section we use a theorem of Heiermann to give a
definition of the indices of inseparability in terms of
Eisenstein polynomials.  We then prove some basic facts
about these indices that we will need later.

Let $K$ be a local field with residue characteristic
$p$ and let $L/K$ be a totally ramified extension of
degree $n=up^k$, where $p\nmid u$.  Let
$\pi_L$ be a uniformizer for $L$ whose minimum
polynomial over $K$ is the Eisenstein polynomial
\begin{equation}
\label{pi_L}
f(X) =X^{n}  -c_1 X^{n-1} +\cdots + (-1)^{n-1}c_{n-1}X + (-1)^{n}c_{n}.
\end{equation}
We wish to define certain invariants of $L/K$ in terms
of the valuations of the coefficients of $f(X)$.  For
$0 \leqslant j \leqslant k$ set
\begin{equation} 
\label{i_jtilde}
\widetilde{\imath}_j=\min\{nv_K(c_h) -h: 1 \leqslant h \leqslant n, v_p (h) \leqslant j\}.
\end{equation}
Let $e_K=v_K(p)$ be the absolute ramification index of
$K$ and set
\begin{equation}
\label{i_j}
i_j =\min\{\widetilde{\imath}_{j'} +(j'-j)ne_K: \, j \leqslant j' \leqslant k \}.  
\end{equation} 
Then $\widetilde{\imath}_j$ may depend on our choice of
$\pi_L$, but $i_j$ is an invariant of the extension
$L/K$:
 
\begin{theorem}

\label{well-defined}
The value of $i_j$ does not depend on the choice of
uniformizer $\pi_L$ for $L$.
\end{theorem}

\begin{proof} 
See Theorem~7.1 of \cite{heiermann}, or Proposition~2.4
and Remark~2.5 of \cite{towers}.
\end{proof}

\begin{definition}
\label{def-i_j}
For $0\leqslant j \leqslant k$ define the $j$th
\emph{index of inseparability} of $L/K$ to be the
nonnegative integer $i_j$ given by \eqref{i_j}.
\end{definition}

If $K$ has characteristic $p$ then $e_K=\infty$, and
hence $i_j=\widetilde{\imath}_j$.  In general we have
\begin{equation} \label{order}
0=i_k < i_{k-1} \leqslant \cdots \leqslant i_1
\leqslant i_0 <\infty.
\end{equation}
For $0 \leqslant j \leqslant k$ and $x\geqslant0$ set
$\widetilde{\phi}_{j}(x)=i_j +p^j x$. 
The generalized Hasse-Herbrand functions $\phi_j: [0, \infty) \to [0, \infty)$ are defined by 
\begin{equation}
\label{varphi_j}
 \phi_j (x) =\min\{ \widetilde{\phi}_{j_0}(x):  \, 0 \leqslant j_0 \leqslant j\}.
\end{equation} 
Let $\phi_{L/K}:[0,\infty)\to[0,\infty)$ denote the
classical Hasse-Herbrand function of $L/K$, as defined
in the appendix to \cite{deligne}.  In Corollary~6.11 of
\cite{heiermann} it is shown that
$\phi_{L/K}(x)=\frac1n\cdot\phi_k(x)$ for $x\geqslant0$.

Let $1\leqslant h \leqslant n$ and set $j=v_p(h)$.
Following \cite{perturb} we define a function $\rho_h:
\mathbb{N} \to \mathbb{N}$ by
\begin{equation} \label{rhoh}
\rho_h(\ell) =\left \lceil \frac{\phi_j(\ell) +h}{n}\right \rceil.
\end{equation}
Let $\pi_L$ be a uniformizer for $L$ whose minimum
polynomial is given by \eqref{pi_L} and let $\pitilde_L$
be a uniformizer for $L$ with minimum polynomial
\begin{equation} \label{pi-tilde}
\ftilde(X)=X^n -\ctilde_{1}X^{n-1} +\cdots
+(-1)^{n-1}\ctilde_{n-1}X +(-1)^{n}\ctilde_n.
\end{equation}

\begin{definition}
\label{l-equiv} 
For $\ell\geqslant1$ we say $f\sim_{\ell}\ftilde$
if $v_K(\ctilde_h -c_h) \geqslant \rho_h(\ell)$
for $1\leqslant h\leqslant n$. 
\end{definition} 

We now specialize the results of this section to the
cases that interests us, namely those where $n=p^k$ and
$L/K$ has two distinct indices of inseparability.  We
first observe that in these cases $L/K$ has a unique
ramification break.  (See the appendix to \cite{deligne}
for definitions and basic properties of ramification
breaks in finite separable extensions.)

\begin{proposition} \label{break}
Let $L/K$ be a totally ramified extension of degree
$p^k$ with two distinct indices of inseparability.  Then
\begin{enumerate}[(i)]
\item The indices of inseparability of $L/K$ satisfy
$i_0=i_1=\cdots=i_{k-1}$ and $i_k=0$. 
\item $L/K$ has a unique upper and lower ramification
break $\B=\displaystyle\frac{i_0}{p^k-1}$.
\end{enumerate}
\end{proposition} 

\begin{proof}
Statement (i) follows from (\ref{order}).  To prove
statement (ii), we use (i) to show that
\begin{align} \nonumber
\phi_{L/K}(x)&=p^{-k}\phi_k(x) \\
&=p^{-k}\cdot\min\{i_0+x,p^kx\} \nonumber \\
&=\begin{cases}
x&\text{if }x\leqslant\B \\
\frac{x+i_0}{p^k}&\text{if }x\geqslant\B
\end{cases} \label{phiLK}
\end{align}
for all $x\geqslant0$.  It follows that $\B$ is the
unique upper and lower ramification break of $L/K$.
\end{proof}

The functions $\rho_h$ given in (\ref{rhoh}) are easy to
describe in this setting.  Let $1\leqslant h\leqslant
p^k-1$ and set $j=v_p(h)$.  Then $v_p(h)\leqslant k-1$, so we
have $\phi_j(\ell)=i_0+\ell$.  Therefore
\[\rho_h(\ell)=\left\lceil\frac{i_0+\ell+h}{p^k}
\right\rceil.\]
For $h=p^k$ we get $\phi_k(\ell)=p^k\phi_{L/K}(\ell)$.
Hence by (\ref{rhoh}) and (\ref{phiLK}) we have
\begin{equation} \label{rhopk}
\rho_{p^k}(\ell)=\lceil\phi_{L/K}(\ell)\rceil+1
=\begin{cases}
\ell+1&\text{if }\ell\leqslant\B \\
\left\lceil\frac{i_0+\ell}{p^k}\right\rceil+1
&\text{if }\ell\geqslant\B.
\end{cases}
\end{equation}

We will need the following basic facts:

\begin{lemma}
\label{lemma0}
If $k\geqslant 2$ and $L/K$ has two distinct indices of
inseparability then $i_{j} =\widetilde{\imath}_j$ for
all $0\leqslant j \leqslant k$.
\end{lemma}

\begin{proof}
First observe that $i_k = \widetilde{\imath}_k =0$. 
Suppose $i_{j} \neq \widetilde{\imath}_j$ for some
$0\leqslant j \leqslant k-2$. Then $i_j= i_{j+1} +p^k
e_K > i_{j+1}=i_j$, a contradiction. Therefore
$i_j =\widetilde{\imath}_j$ for all $j\not=k-1$.
In particular, since $k\geqslant2$ we
have $i_0=\widetilde{\imath}_0$.  We need to show that
$i_{k-1} = \widetilde{\imath}_{k-1}$. Suppose not. Then
$i_{k-1} =i_{k} +p^k e_K =p^k e_K$.  Hence by
Proposition~\ref{break}(i) we have $p^k e_K=i_{k-1} =i_0
=\widetilde{\imath}_0$.  This is a contradiction since
$\widetilde{\imath}_0$ is not divisible by $p$.  Thus
$i_{k-1}=\widetilde{\imath}_{k-1}$.  We conclude that $i_{j}
= \widetilde{\imath}_j$ for all $0\leqslant j \leqslant k$. 
\end{proof}

\begin{lemma} \label{i0bound}
If $k\geqslant 2$ and $L/K$ has two distinct indices of
inseparability then $i_0<p^ke_K$.
\end{lemma}

\begin{proof}
By Lemma~\ref{lemma0}, Proposition~\ref{break}(i), and
(\ref{i_j}) we have
\[\widetilde{\imath}_0=i_0= i_{k-1}
\leqslant i_k+p^ke_K=p^ke_K.\]
Since $p\nmid\widetilde{\imath}_0$ the inequality is
strict.
\end{proof}

\section{Perturbing uniformizers of $L$} \label{quote}

We will make frequent use of two theorems from
\cite{perturb}.  First we state the ``stability
theorem''.

\begin{theorem}[\cite{perturb}, Theorem 4.3]
\label{stability}
Let $\pi_L,\pitilde_L$ be uniformizers for $L$ with
minimum polynomials $f(X)$, $\ftilde(X)$ given by
(\ref{pi_L}), (\ref{pi-tilde}).  Assume there is
$\ell\geqslant1$ such that $\pitilde_L\equiv\pi_L
\pmod{\M_L^{\ell+1}}$.  Then $f\sim_{\ell}\ftilde$.
\end{theorem}

Theorem~\ref{stability} gives a relationship between the
coefficients of the minimum polynomial for $\pi_L$ and
the coefficients of the minimum polynomial of its
perturbation $\pitilde_L$, essentially saying
that if the perturbation is small (that is, if
$\pitilde_L -\pi_L$ has large valuation) then the
coefficients of the corresponding minimum polynomials
are close (that is, the difference $\ctilde_h -c_h$ has
large valuation).  The next result, the ``perturbation
theorem'', refines Theorem~\ref{stability} for certain
special values of $h$. Theorems~\ref{stability} and
\ref{perturbation} form the basis of the iterative
procedure which we use to prove our main result.

\begin{theorem}[\cite{perturb}, Theorem 4.5]
\label{perturbation} Let $K$ be a local field whose
residue field has characteristic $p$ and let $L/K$ be a
finite separable totally ramified extension of degree
$n=up^k$, with $p\nmid u$.  For $0 \leqslant j \leqslant k$ write the
$j$th index of inseparability of $L/K$ in the form
$i_j= A_j n -b_j$ with $1\leqslant b_j \leqslant n$. Let
$\pi_L, \pitilde_L$ be uniformizers for $L$ such that
there are $\ell \geqslant 1$ and $r\in \OO_K$ with
$\pitilde_L \equiv \pi_L +r \pi_L^{\ell +1}
\pmod{\M_L^{\ell +2}}$. Let $0\leqslant j
\leqslant k$ satisfy $\min\{v_p(\phi_j(\ell)), k\} =j$,
and let $h$ be the unique integer such that $1\leqslant
h \leqslant n$ and $n$ divides $\phi_j(\ell) +h$. Set
$t=(\phi_j(\ell) +h)/n$ and $h_0 =h/p^j$. Then 
\[ \ctilde_h \equiv c_h +\sum_{m \in S_j} g_m
c_n^{t-A_m}c_{b_m}r^{p^m} \pmod{\M_K^{t+1}},\] 
where 
\[ S_j =\{ m: \, 0 \leqslant m \leqslant j, \, \phi_j(\ell)=\widetilde{\phi}_m(\ell)\} \]
and 
\[ g_m=
\begin{cases}
(-1)^{t+\ell+A_m}(h_0 p^{j-m}+\ell -up^{k-m})
&\text{if } 1\leqslant b_m <h, \\
(-1)^{t+\ell+A_m}(h_0p^{j -m} +\ell) 
& \text{if } h \leqslant b_m < n,\\
(-1)^{t+\ell +A_m}up^{k-m}&\text{if } b_m=n. 
\end{cases} \]
\end{theorem}

We can specialize Theorem~\ref{perturbation} to our
setting:

\begin{proposition} \label{simple}
Let $K$ be a local field whose residue field
$\overline{K}$ has characteristic $p$ and let $L/K$ be a
finite totally ramified extension of degree $p^k$ which
has two distinct indices of inseparability.  Write
$i_0=p^kA_0-b_0$ with $1\leqslant b_0\leqslant p^k$.
Let $\pi_L, \pitilde_L$ be uniformizers for $L$
such that there are $\ell \geqslant 1$ and $r\in \OO_K$
such that $\pitilde_L \equiv \pi_L +r
\pi_L^{\ell +1} \pmod{\M_L^{\ell +2}}$.  Recall from
Proposition~\ref{break}(ii) that $\B=\dfrac{i_0}{p^k-1}$ is
the unique upper and lower ramification break of $L/K$.
\begin{enumerate}[(i)]
\item If $\ell<\B$ then
\[\ctilde_{p^k}\equiv c_{p^k}+c_{p^k}^{\ell+1}r^{p^k}
\pmod{\M_K^{\ell+2}}.\]
\item If $p^k-1\mid i_0$ and $\ell=\B$ then
\[\ctilde_{p^k}\equiv
c_{p^k} +(-1)^{A_0+1}b_0c_{p^k}^{\ell+1-A_0}c_{b_0}r
+c_{p^k}^{\ell+1}r^{p^k}\pmod{\M_K^{\ell+2}}.\]
\item Suppose $\ell>\B$ or
$\ell\not\equiv b_0\pmod{p^k}$.  Let $h$ be the unique
integer such that $1\leqslant h \leqslant p^k$ and $p^k$
divides $i_0+\ell+h$.  Set $t=(i_0+\ell+h)/p^k$.  Then
\[\ctilde_h\equiv
c_h +(-1)^{t+\ell+A_0}b_0c_{p^k}^{t-A_0}c_{b_0}r
\pmod{\M_K^{t+1}}.\]
\end{enumerate}
\end{proposition}

\begin{proof}
By Proposition~\ref{break}(i) we have $i_k=0$ and
$i_0=i_1=\cdots=i_{k-1}$.  Hence $\phi_j(x)=i_0+x$ for
all $0\leqslant j\leqslant k-1$, and
\[\phi_k(x)=\begin{cases}p^kx
&\text{if }0\leqslant x\leqslant\B, \\
i_0+x&\text{if }x\geqslant\B.\end{cases}\]
(i) If $\ell<\B$ then $\phi_k(\ell)=p^k\ell$, and hence
$\min\{v_p(\phi_k(\ell)),k\}=k$.  Therefore we
may apply Theorem~\ref{perturbation} with $n=p^k$,
$j=k$, $h=p^k$, $h_0=1$, $t=\ell+1$, and $S=\{k\}$.
Since $A_k=1$ and $b_k=p^k$ we get
$g_k=(-1)^{\ell+1+\ell+1}p^{k-k}=1$, and hence
\[\ctilde_{p^k}\equiv
c_{p^k}+1\cdot c_{p^k}^{\ell}\cdot c_{p^k}\cdot r^{p^k}
\pmod{\M_K^{\ell+2}}.\]
(ii) If $\ell=\B$ then $\min\{v_p(\phi_k(\ell)),k\}=k$,
so we may apply Theorem~\ref{perturbation} with $n=p^k$,
$j=k$, $h=p^k$, $h_0=1$, $t=\ell+1$, and $S=\{0,k\}$.
Therefore
\begin{alignat*}{2}
g_0&\equiv(-1)^{\ell+1+\ell+A_0}\ell
\equiv(-1)^{A_0+1}b_0&&\pmod{p^k} \\
g_k&=(-1)^{\ell+1+\ell+1}p^{k-k}=1 \\
\ctilde_{p^k}&\equiv
c_{p^k} +(-1)^{A_0+1}b_0\cdot c_{p^k}^{\ell+1-A_0}\cdot
c_{b_0}\cdot r
+1\cdot c_{p^k}^{\ell}\cdot c_{p^k}\cdot r^{p^k}
&&\pmod{\M_K^{\ell+2}}.
\end{alignat*}
(iii) Set $j=\min\{v_p(i_0+\ell),k\}$.  If
$\ell>\B$ then
$\phi_j(\ell)=i_0+\ell$.  If $\ell\not\equiv
b_0\pmod{p^k}$ then $j<k$, so once again we have
$\phi_j(\ell)=i_0+\ell$.  In either case we get
$j=\min\{v_p(\phi_j(\ell)),k\}$.  Therefore the
hypotheses of Theorem~\ref{perturbation} are satisfied
with $n=p^k$, $t=(i_0+\ell+h)/p^k$, and $S=\{0\}$.  It
follows that
\begin{alignat*}{2}
g_0&\equiv(-1)^{t+\ell+A_0}(h+\ell)
\equiv(-1)^{t+\ell+A_0}b_0&&\pmod{p^k} \\
\ctilde_h&\equiv
c_h +(-1)^{t+\ell-A_0}b_0\cdot c_{p^k}^{t-A_0}\cdot
c_{b_0}\cdot r&&\pmod{\M_K^{t+1}}.
\end{alignat*}
\end{proof}

\section{The main theorem}
\label{iteration}

In this section we state our main result.  Let $K$ be a
local field with perfect residue field $\Kbar$ and let
$L/K$ be a totally ramified subextension of $\Omega/K$
of degree $n=p^k$.  Let $\pi_L$ be a uniformizer for $L$
and let
\begin{equation} \label{piLpk}
f(X) =X^{p^k}  -c_1 X^{p^k-1} +\cdots +
(-1)^{p^k-1}c_{p^k-1}X + (-1)^{p^k}c_{p^k}
\end{equation}
be the minimum polynomial of $\pi_L$ over $K$.  Assume
that $L/K$ has only two distinct indices of
inseparability.  Then $i_0=i_1=\cdots=i_{k-1}>i_k=0$ by
Proposition~\ref{break}(i).  If $k\geqslant2$ then by
Lemma~\ref{lemma0} we have $i_{j}=\widetilde{\imath}_j$
for all $0\leqslant j \leqslant k$.  This conclusion
need not hold when $K$ has characteristic $0$
and $k=1$; for instance, if $f(X)=X^p-\pi_K$ then
$\widetilde{\imath}_0=\infty$ and $i_0=pe_K$.  Since
the case $k=1$ was considered in \cite{amano}, we will
assume $k\geqslant 2$.

Under these hypotheses we will show that there is a
uniformizer $\pi_L$ for $L$ whose minimum polynomial
over $K$ has only three nonzero terms.  It is not hard
to see that at least three terms are needed. Indeed, let
$\pi_L$ be a uniformizer for $L$ whose minimum
polynomial is given by \eqref{piLpk}.  By
Lemma~\ref{lemma0} and (\ref{i_jtilde}) we have
\[i_0=\widetilde{\imath}_0 
=\min\{ p^k v_K(c_h) -h:1\leqslant h\leqslant p^k,\; v_p(h) =0\}. \]
Let $1\leqslant b_0\leqslant p^k$ be such that
$b_0\equiv-i_0\pmod{p^k}$.  Then $i_0=p^kA_0-b_0$ with
$v_K(c_{b_0})=A_0=p^{-k}(i_0+b_0)$.  It follows that
$f(X)$ has at least three nonzero
terms, namely $X^{p^k}$, $(-1)^{b_0}c_{b_0}X^{p^k -b_0}$,
and $(-1)^{p^k} c_{p^k}$, and moreover the valuation of
$c_{b_0}$ is independent of the choice of uniformizer
$\pi_L$. To show that three terms suffice, the strategy
is to construct a Cauchy sequence of uniformizers of $L$
that converge to a uniformizer whose Eisenstein
polynomial has three terms.

Let $\omega\in\RR\smallsetminus\{0\}$.  Keeping in mind
that $i_0=p^kA_0-b_0$, we define
$\psi_{i_0,\omega}:\OO_K \to \OO_K$ by
\[\psi_{i_0,\omega}(x) =x^{p^k}-(-1)^{A_0} b_0 \omega x.\]
Then $\psi_{i_0,\omega}$ induces an additive group
homomorphism $\psibar_{i_0,\omega}: \Kbar \to \Kbar$,
given by
\[\psibar_{i_0,\omega}(x)
=x^{p^k}-(-1)^{A_0}\overline{b_0\omega}x.\]
Let $\Tbar_{i_0,\omega}\subset\Kbar$ be a set
of coset representatives for
$\Kbar/\psibar_{i_0,\omega}(\Kbar)$ and let
$T_{i_0,\omega}\subset\RR$ be such that
$\Tbar_{i_0,\omega}=\{x+\M_K:x\in T_{i_0,\omega}\}$.  In
the cases where $\psibar_{i_0,\omega}(\Kbar)=\Kbar$ we
take $T_{i_0,\omega}=\{0\}$.

\begin{definition} \label{standard}
Let $f(X)\in\OO_K[X]$ be an Eisenstein
polynomial of degree $p^k$.  We say that $f(X)$ is in
\emph{standard form} if there exist $1\leqslant A_0\leqslant e_K$,
$1\leqslant b_0\leqslant p^k$ with $p\nmid b_0$, and $\omega \in \mathfrak{R} \smallsetminus \{0\}$ such that
\begin{equation} \label{standardf}
f(X)=X^{p^k}+(-1)^{b_0}\omega\pi_K^{A_0}X^{p^k-b_0}
+(-1)^{p^k}a\pi_K
\end{equation}
for some $a\in\OO_K$ which can be expressed as
follows.  Set $i_0=p^kA_0-b_0$ and
$\B=\displaystyle\frac{i_0}{p^k-1}$.  Then
\begin{enumerate}[(i)]
\item If $\B \not \in \mathbb{Z}$ there
are $\alpha_j \in \RR$ with $\alpha_j =0$ for
all $j \equiv b_0 \pmod{p^k}$ such that
\[a=1+\sum_{j=1}^{\lfloor\B\rfloor}\alpha_j
\pi_K^{j}.\]
\item If $\B\in \mathbb{Z}$ there are
$\gamma\in T_{i_0,\omega}$ and $\alpha_j \in \RR$ with
$\alpha_j =0$ for all $j \equiv b_0 \pmod{p^k}$ such
that
\[a=1+\left(\sum_{j=1}^{\B-1}\alpha_j
\pi_K^{j}\right) +\gamma \pi_K^{\B}.\]
\end{enumerate}
Let $\ell\geqslant1$.  We say that $f(X)$ is in
\emph{$\ell$-standard form} if there is a degree-$p^k$ Eisenstein
polynomial $\fhat(X)\in\OO_K[X]$
in standard form such that $f\sim_{\ell}\fhat$ (see Definition~\ref{l-equiv}).
\end{definition}

\begin{remark} \label{exclude}
Let $f(X)\in\OO_K[X]$ be a polynomial of the form
(\ref{standardf}), with
$\omega\in\RR\smallsetminus\{0\}$.  Let $\pi_L\in\Omega$
be a root of $f(X)$ and set $L=K(\pi_L)$.  If
$p\nmid b_0$ and $A_0\leqslant e_K$ then it follows from
(\ref{i_j}) that the indices of inseparability of $L/K$
are given by
$i_k=0$ and $i_j=p^kA_0-b_0$ for $0\leqslant j\leqslant k-1$.  On
the other hand, if $p\nmid b_0$ and $A_0>e_K$ then
$i_{k-1}=p^ke_K$ and $i_0=p^kA_0-b_0$, while if
char$(K)=0$ and $p\mid b_0$ then $i_0=i_1+p^ke_K$.  In
both of these cases we have $0=i_k<i_{k-1}<i_0$, so we
have at least three distinct indices of inseparability.
Since we are interested in extensions with just two
indices of inseparability, we exclude these cases in the
definition above.
\end{remark}

\begin{remark}
Suppose that $\Kbar$ is finite.  If
$\psibar_{i_0,\omega}$ has no
nonzero roots in $\Kbar$ then
$\psibar_{i_0,\omega}(\Kbar)=\Kbar$.  On the
other hand, if $\psibar_{i_0,\omega}$ has
a nonzero root $\overline{z}\in\Kbar$ then
$\overline{(-1)^{A_0}b_0 \omega}=\overline{z}^{p^k-1}$, so
$\psibar_{i_0,\omega}(x)$ can be rewritten in terms of
$\overline{\kappa}(x)=x^{p^k}-x$ as
\begin{equation}
\label{eq-psi}
 \psibar_{i_0,\omega}(x)=
\overline{z}^{p^k}((x/\overline{z})^{p^k}
-(x/\overline{z}))
= \overline{z}^{p^k}\overline{\kappa}(x/\overline{z}).
 \end{equation}
Hence $\psibar_{i_0,\omega}(\Kbar)
=\overline{z}^{p^k}\overline{\kappa}(\Kbar)$.  Choose a
set $\overline{U}$ of coset representatives for
$\Kbar/\overline{\kappa}(\Kbar)$.  Then
$\overline{z}^{p^k}\overline{U}$ is a set of coset
representatives for
$\Kbar/\psibar_{i_0,\omega}(\Kbar)$.  Thus we
only need to choose a single set of coset
representatives in the case where $\Kbar$ is finite,
although we do need to pick a nonzero root of
$\psibar_{i_0,\omega}$ for each pair
$(i_0,\omega)$ such that
$\psibar_{i_0,\omega}(\Kbar)\not=\Kbar$.
\end{remark}

We now state our main result.

\begin{theorem} \label{main}
Let $K$ be a local field with residue characteristic $p$
and let $k\geqslant 2$.  Then there is a one-to-one
correspondence between the set of degree-$p^k$
Eisenstein polynomials $f(X)\in\OO_K[X]$ in standard
form, and the set of $K$-isomorphism classes of totally
ramified separable extensions $L/K$ of degree $p^k$
which have two distinct indices of inseparability.  If
$f(X)$ is in standard form and $\pi\in\Omega$ is a root
of $f(X)$ then $K(\pi)/K$ is an extension corresponding
to $f(X)$.
\end{theorem}

\section{Proof of the theorem: Existence}
\label{existence}

In this section we show that every totally ramified
separable extension $L/K$ of degree $p^k$ which has two
distinct indices of inseparability is generated by a
root of an Eisenstein polynomial $f(X)\in\OO_K[X]$ in
standard form.  In the next section we show that $f(X)$
is uniquely determined by $L/K$, thus completing the
proof of Theorem~\ref{main}.

Let $\pi_K$ be a fixed uniformizer of $K$ and choose an
arbitrary uniformizer $\pi_L$ for $L$, with minimum
polynomial $f(X)$ given by \eqref{piLpk}.  Let
$\beta\in\RR$ be such that
$c_{p^k}\equiv\beta\pi_K\pmod{\M_K^2}$.  Since $\Kbar$
is perfect, there exists $\xi \in \RR$  such that
$\xi^{p^k}=\beta$.  Then $\xi^{-1}\pi_L$ is a root of
\begin{equation}
\label{pi_L'}
\fhat(X) =X^{p^k}  -\chat_1 X^{p^k-1} +\cdots +
(-1)^{p^k-1}\chat_{p^k-1}X + (-1)^{p^k}\chat_{p^k},
\end{equation}
where $\chat_j =\xi^{-j}c_j$.  In particular, we have
$\chat_{p^k}=\beta^{-1}c_{p^k}\equiv\pi_K\pmod{\M_K^2}$.
Hence we may assume that the constant term in
\eqref{piLpk} satisfies
$c_{p^k}\equiv\pi_K \pmod{\M_K^2}$.
Since $v_K(c_{b_0})=A_0$ there is
$\omega\in\RR\smallsetminus\{0\}$ such that
$c_{b_0}\equiv\omega\pi_K^{A_0}\pmod{\M_K^{A_0+1}}$.  We
will show in Proposition~\ref{lemma4.6} that, once
$\pi_K$ has been fixed, $\omega$ is an invariant of the
extension $L/K$.

To show that every
isomorphism class of totally ramified extensions of
degree $p^k$ with two indices of inseparability comes
from a polynomial in standard form we use induction on
$\ell$ to show that every such extension comes from a
polynomial in $\ell$-standard form.  The base case is
given by the following lemma:

\begin{lemma}
\label{lemma1}
Let $\pi_L$ be a uniformizer of $L$ whose minimum
polynomial \eqref{piLpk} satisfies
$c_{p^k}\equiv\pi_K\pmod{\M_K^2}$.
Let $\omega\in\RR$ be such that
$c_{b_0}\equiv\omega\pi_K^{A_0}\pmod{\M_K^{A_0+1}}$, and
set
\[g(X)=X^{p^k}+(-1)^{b_0}\omega\pi_K^{A_0}X^{p^k-b_0}
+(-1)^{p^k}\pi_K.\]
Then $f\sim_1g$ (see Definition~\ref{l-equiv}). 
\end{lemma}

\begin{proof}
Suppose first that $1\leqslant h < b_0$.  Then
$v_p(h)<k$, so
\[ \rho_h(1) =\left \lceil \frac{i_0 +1 +h}{p^k}
\right \rceil  = \left \lceil \frac{p^kA_0-b_0 +h +1}{p^k}
\right \rceil =A_0.\]
By Lemma~\ref{lemma0} and Proposition~\ref{break}(i) we
have
\[\widetilde{\imath}_{k-1}=i_{k-1}=i_0=p^kA_0-b_0.\]
Hence by (\ref{i_jtilde}) we get
$v_K(c_h) \geqslant v_K(c_{b_0})=A_0=\rho_h(1)$.
On the other hand, if $b_0<h<p^k$ then $v_p(h)<k$, and
hence
\[ \rho_h(1)= \left \lceil \frac{i_0 +1 +h}{p^k}
\right \rceil  =\left \lceil \frac{p^k A_0 -b_0 +1+h}{p^k}
\right \rceil  = A_0+1. \]
Therefore $v_K(c_h) \geqslant A_0 +1 = \rho_h(1)$.
Since $c_{b_0} \equiv \omega
\pi_K^{A_0}\pmod{\M_K^{A_0+1}}$ we have
$v_K(c_{b_0} -\omega \pi_K^{A_0}) \geqslant A_0 +1$.  We
also have
\[ \rho_{b_0}(1)=\left \lceil \frac{i_0 +1 +b_0}{p^k}\right \rceil =\left \lceil \frac{p^k A_0 -b_0 +1+b_0}{p^k} \right \rceil =A_0+1.  \]
Hence $v_K(c_{b_0} -\omega \pi_K^{A_0})\geqslant
\rho_{b_0}(1)$.  Since
$c_{p^k}\equiv\pi_K\pmod{\M_K^2}$ we have
$v_K(c_{p^k}-\pi_K)\geqslant2$, and by (\ref{rhopk}) we
have $\rho_{p^k}(1)=2$.  Hence
$v_K(c_{p^k}-\pi_K)\geqslant\rho_{p^k}(1)$.
\end{proof}

We now use Proposition~\ref{simple} to prove the
inductive step:

\begin{proposition} \label{main-thm1}
Let $\ell \geqslant 1$ and let $\pi_L$ be a uniformizer
for $L$ whose minimum polynomial $f(X)$ is in
$\ell$-standard form.  Then there exists
$r\in\RR$ such that the minimum polynomial
$\ftilde(X)$ of $\pitilde_L
=\pi_L +r\pi_{L}^{\ell+1}$ is in $(\ell+1)$-standard
form.
\end{proposition}

\begin{proof}
We begin by making some numerical observations.  Let
$x\ge0$, let $1\leqslant h\leqslant p^k-1$, and set
$j=v_p(h)$.  Then $v_p(h)<k$, so
\begin{equation} \label{phijx}
\frac{\phi_j(x)+h}{p^k}=\frac{i_0+x+h}{p^k}
=\frac{p^kA_0-b_0+x+h}{p^k}.
\end{equation}
Hence by (\ref{rhoh}) we get
\begin{equation} \label{ellplus}
\rho_h(\ell+1)=\begin{cases}
\rho_h(\ell)&\text{ if }\ell\not\equiv b_0-h\pmod{p^k} \\
\rho_h(\ell)+1&\text{ if }\ell\equiv b_0-h\pmod{p^k}.
\end{cases}
\end{equation}
In addition, note that if $\ell=\B\in\Z$ then
$(p^k-1)\ell=i_0=p^kA_0-b_0$, and hence
$\ell\equiv b_0\pmod{p^k}$.

Let
\begin{equation} \label{ftpk}
\ftilde(X) =X^{p^k}  -\ctilde_1 X^{p^k-1} +\cdots +
(-1)^{p^k-1}\ctilde_{p^k-1}X +
(-1)^{p^k}\ctilde_{p^k}
\end{equation}
be the minimum polynomial for
$\pitilde_L=\pi_L +r\pi_{L}^{\ell+1}$ over $K$.  We
wish to choose $r\in\RR$ so that $\ftilde(X)$ is in
$(\ell+1)$-standard form.  We divide the proof into two
cases, based on the congruence class of $\ell$ modulo
$p^k$.

\smallskip \noindent
\textbf{Case 1:}
$\boldsymbol{\ell\not\equiv b_0\pmod{p^k}}$\textbf{.}
In this case there is $1\leqslant h_1 \leqslant p^k-1$
such that $p^k\mid-b_0+\ell+h_1$.  Hence by
(\ref{ellplus}) we have $\rho_{h_1}(\ell+1)
=\rho_{h_1}(\ell)+1$ and $\rho_h(\ell+1)=\rho_h(\ell)$
for all $1\leqslant h\leqslant p^k-1$ such that
$h\not=h_1$.  Set $t=\rho_{h_1}(\ell)$.
Then by Proposition~\ref{simple}(iii) we have
\[\ctilde_{h_1}\equiv c_{h_1}
+(-1)^{t+\ell+A_0}h_1c_{p^k}^{t-A_0}c_{h_1}r
\pmod{\M_K^{t+1}}.\]
Since $c_{p^k}\equiv\pi_K\pmod{\M_K^2}$ and
$c_{h_1}\equiv\omega\pi_K^{A_0}
\pmod{\M_K^{A_0+1}}$ we get
\begin{equation} \label{eq-b0}
\ctilde_{h_1}
\equiv c_{h_1}+(-1)^{t+\ell+A_0}h_1\omega\pi_K^{t}r
\pmod{\M_K^{t+1}}.
\end{equation}

If $h_1\not=b_0$ (equivalently, $p^k\nmid\ell$) then
$v_K(c_{h_1}) \geqslant \rho_{h_1}(\ell)=t$.  Hence
there is $\alpha\in\RR$ such that
$c_{h_1}\equiv\alpha\pi_K^{t}\pmod{\M_K^{t+1}}$.  Choose
$r\in\RR$ to satisfy
\begin{equation} \label{r}
r\equiv(-1)^{t+\ell +A_0+1}\alpha(b_0\omega)^{-1}
\pmod{\M_K}.
\end{equation}
Then by (\ref{eq-b0}) we get
\[v_K(\ctilde_{h_1})\geqslant t+1=
\rho_{h_1}(\ell)+1=\rho_{h_1}(\ell+1).\]
Suppose $h_1=b_0$.  Then by (\ref{eq-b0}) we get
\[ \ctilde_{b_0} -\omega \pi_K^{A_0} \equiv
c_{b_0} -\omega \pi_K^{A_0}
+(-1)^{t+\ell+A_0}b_0\omega\pi_K^{t}r \pmod{\M_K^{t+1}}. \] 
Since $c_{b_0} -\omega \pi_K^{A_0} \in
\M_K^{t}$ there is $\alpha\in\RR$
such that $c_{b_0}-\omega\pi_K^{A_0}\equiv\alpha\pi_K^t
\pmod{\M_K^{t+1}}$.  By choosing $r\in\RR$ to satisfy
(\ref{r}) we get
\[v_K(\ctilde_{b_0}-\omega\pi_K^{A_0})\geqslant
\rho_{b_0}(\ell)+1=\rho_{b_0}(\ell+1).\] 

Let $1\leqslant h\leqslant p^k$ with $h \neq h_1$.
By Theorem~\ref{stability} we have
$v_K(\ctilde_{h}-c_{h}) \geqslant \rho_{h}(\ell)$.
If $h\not=b_0$ and $h\not=p^k$ then
$v_K(c_h)\geqslant\rho_{h}(\ell)$, and hence
$v_K(\ctilde_{h})\geqslant\rho_{h}(\ell)
=\rho_{h}(\ell+1)$.  In the case $h=b_0$ we have
$v_K(c_{b_0}-\omega\pi_K^{A_0})\geqslant
\rho_{b_0}(\ell)$, and hence
\[v_K(\ctilde_{b_0} - \omega \pi_K^{A_0})\geqslant
\rho_{b_0}(\ell)=\rho_{b_0}(\ell +1).\]
Finally, when $h=p^k$ we have $v_K(c_{p^k}-\pi_K a)\geqslant
\rho_{p^k}(\ell)$ for some $a$ satisfying conditions (i)
and (ii) of Definition~\ref{standard}.  It follows that
$v_K(\ctilde_{p^k}-\pi_K a)\geqslant \rho_{p^k}(\ell)$.
Suppose $\ell>\B$.  Since $\ell\not\equiv b_0\pmod{p^k}$
it follows from (\ref{rhopk}) that
$\rho_{p^k}(\ell+1)=\rho_{p^k}(\ell)$.  Hence
$v_K(\ctilde_{p^k}-\pi_K a)\geqslant\rho_{p^k}(\ell+1)$.
Suppose $\ell\leqslant \B$.  Then by (\ref{rhopk}) we have
$\rho_{p^k}(\ell)=\ell+1$ and
$\rho_{p^k}(\ell+1)=\ell+2$.  Hence there is
$\alpha_{\ell}\in\B$ such that
\[\ctilde_{p^k}-\pi_K a\equiv
\alpha_{\ell}\pi_K^{\ell+1}\pmod{\M_K^{\ell+2}}.\]
Set $\atilde=a+\alpha_{\ell}\pi_K^{\ell}$.  Then
$v_K(\ctilde_{p^k}-\pi_K\atilde)\geqslant
\ell+2=\rho_{p^k}(\ell+1)$.  Since
$\ell\not\equiv b_0\pmod{p^k}$ we have $\ell\not=\B$.
Hence $\atilde$ satisfies conditions (i) and (ii) of
Definition~\ref{standard}.  We conclude that if
$\ell\not\equiv b_0\pmod{p^k}$ then $\ftilde(X)$ is in
$(\ell+1)$-standard form.

\smallskip \noindent
\textbf{Case 2:}
$\boldsymbol{\ell\equiv b_0\pmod{p^k}}$\textbf{.}
By (\ref{ellplus}) we have
$\rho_{h}(\ell+1)=\rho_{h}(\ell)$ for
$1\leqslant h \leqslant p^k-1$.  Reasoning as in the
preceding paragraph we find that for any choice of
$r\in\RR$ the coefficients of the minimum polynomial
(\ref{ftpk}) of $\pitilde_L=\pi_L+r\pi_L^{\ell+1}$ over
$K$ satisfy $v_K(\ctilde_{b_0} -\omega\pi_K^{A_0})
\geqslant \rho_{b_0}(\ell+1)$ and
$v_K(\ctilde_{h})\geqslant \rho_{h}(\ell+1)$ for
all $1\leqslant h\leqslant p^k-1$ such that $h\not=b_0$.
We need to choose $r$ so that $v_K(\ctilde_{p^k}
-\pi_K\atilde)\geqslant\rho_{p^k}(\ell+1)$ for
some $\atilde$ which satisfies conditions (i) and
(ii) of Definition~\ref{standard}.

If $\ell < \B$ then by Proposition~\ref{simple}(i) we get
\[\ctilde_{p^k} \equiv c_{p^k}+c_{p^k}^{\ell+1}
r^{p^k}\pmod{\M_K^{\ell+2}}.\]
Since $c_{p^k}\equiv\pi_K\pmod{\M_K^2}$ it follows that
\begin{equation} \label{cpktilde}
\ctilde_{p^k} \equiv c_{p^k} +\pi_K^{\ell+1}r^{p^k}
\pmod{\M_K^{\ell +2}}.
\end{equation}
Write
\begin{equation}
\label{cpk}
c_{p^k}=\pi_K(1+ \alpha_1 \pi_K +\alpha_2 \pi_K^{2}+\cdots)
\end{equation}
with $\alpha_i \in \RR$. Since the residue field $\Kbar$
is perfect, there exists $r\in \RR$ such that $r^{p^k}
\equiv -\alpha_{\ell} \pmod{\M_K}$.  It follows from
(\ref{cpktilde}) that with this choice of $r$ we get
\[\ctilde_{p^k} \equiv \pi_K(1+ \alpha_1 \pi_K
+\cdots + \alpha_{\ell-1}\pi_K^{\ell-1} + 0\cdot
\pi_K^{\ell}) \pmod{\M_K^{\ell+2}}.\]
By (\ref{rhopk}) we have $\rho_{p^k}(\ell)=\ell+1$ and
$\rho_{p^k}(\ell+1)=\ell+2$.  Since $f(X)$ is in
$\ell$-standard form, it follows that $\ftilde(X)$ is in
$(\ell+1)$-standard form.

If $\ell=\B$ then by Proposition~\ref{simple}(ii)
we get
\[\ctilde_{p^k}\equiv
c_{p^k} +(-1)^{A_0+1}b_0c_{p^k}^{\ell+1-A_0}c_{b_0}r
+c_{p^k}^{\ell+1}r^{p^k}\pmod{\M_K^{\ell+2}}.\]
Since $c_{p^k}\equiv\pi_K\pmod{\M_K^2}$ and
$c_{b_0}\equiv\omega\pi_K^{A_0}\pmod{\M_K^{A_0+1}}$
we can rewrite this congruence as
\begin{alignat}{2} \nonumber
\ctilde_{p^k}
&\equiv c_{p^k} +(-1)^{A_0+1}b_0\omega\pi_K^{\ell+1}r
+\pi_K^{\ell+1}r^{p^k}
&&\pmod{\M_K^{\ell+2}} \\
&\equiv c_{p^k}+\pi_K^{\ell+1}\psi_{i_0,\omega}(r)
&&\pmod{\M_K^{\ell+2}}. \label{psi}
\end{alignat}
Write $c_{p^k}$ as a power series in $\pi_K$ with
coefficients in $\RR$, as in (\ref{cpk}).  It
follows from the definition of $\Tbar_{i_0,\omega}$ that
there are $\gamma\in T_{i_0,\omega}$ and $r\in\RR$ such
that $\alpha_{\ell}\equiv\gamma-\psi_{i_0,\omega}(r)
\pmod{\M_K}$.  Hence by (\ref{psi}) we get
\[\ctilde_{p^k} \equiv \pi_K(1+
\alpha_1 \pi_K +\cdots + \alpha_{\ell-1}\pi_K^{\ell-1} +
\gamma\cdot \pi_K^{\ell}) \pmod{\M_K^{\ell+2}}.\]
Once again we get $\rho_{p^k}(\ell)=\ell+1$ and
$\rho_{p^k}(\ell+1)=\ell+2$ by (\ref{rhopk}).
Since $f(X)$ is in $\ell$-standard form it follows that
$\ftilde(X)$ is in $(\ell+1)$-standard form.

Suppose $\ell > \B$.  Since
$\ell\equiv b_0\pmod{p^k}$ it follows from
(\ref{rhopk}) that
\begin{align*}
\rho_{p^k}(\ell)&=\frac{i_0+\ell}{p^k}+1
=\frac{i_0+\ell+p^k}{p^k} \\
\rho_{p^k}(\ell+1)
&=\left\lceil\frac{i_0+\ell+1}{p^k}\right\rceil+1
=\rho_{p^k}(\ell)+1.
\end{align*}
Set $t=\rho_{p^k}(\ell)$.  Then by
Proposition~\ref{simple}(iii) we get
\[ \ctilde_{p^k} \equiv c_{p^k} +(-1)^{t +\ell+A_0}b_0
c_{p^k}^{t-A_0}c_{b_0}r \pmod{\M_K^{t+1}}. \]
Since $c_{p^k}\equiv\pi_K\pmod{\M_K^2}$ and
$c_{b_0}\equiv\omega\pi_K^{A_0}\pmod{\M_K^{A_0+1}}$
we can rewrite this congruence as
\[\ctilde_{p^k}
\equiv c_{p^k} +(-1)^{t+\ell+A_0}b_0\omega\pi_K^tr
\pmod{\M_K^{t+1}}.\]
Write $c_{p^k}$ as in (\ref{cpk}) and let $r\in\RR$
satisfy
\[r\equiv(-1)^{t+\ell+A_0+1}(b_0\omega)^{-1}\alpha_t
\pmod{\M_K}.\]
Then
\[\ctilde_{p^k} \equiv \pi_K(1+ \alpha_1 \pi_K
+\cdots + \alpha_{t-1}\pi_K^{t-2} + 0\cdot
\pi_K^{t-1}) \pmod{\M_K^{t+1}}.\]
Since $f(X)$ is in $\ell$-standard form and
$\rho_{p^k}(\ell+1)=t+1$ it follows that $\ftilde(X)$ is
in $(\ell+1)$-standard form.
\end{proof}

We now use induction to show that there is a uniformizer
$\pihat_L$ for $L$ whose minimum polynomial $\fhat(X)$
is in standard form.  Indeed, it follows from
Lemma~\ref{lemma1} that the minimum polynomial
$f(X)=f_1(X)$ of the initial uniformizer
$\pi_L=\pi_L^{(1)}$ is in 1-standard form.
Proposition~\ref{main-thm1} shows that we can find a new
uniformizer $\pi_L^{(2)}$ whose minimum polynomial
$f_2(X)$ is in 2-standard form.  Repeating this
procedure, at step $\ell$ we get a new uniformizer
$\pi_L^{(\ell)}$ whose minimum polynomial $f_{\ell}(X)$
is in $\ell$-standard form.  The sequence of
uniformizers $(\pi_L^{(\ell)})_{\ell\geqslant1}$ is
Cauchy, and hence converges to a uniformizer
$\pihat_L$ for $L$.  It follows from
Theorem~\ref{stability} that the sequence
$(f_{\ell}(X))_{\ell\geqslant1}$ of minimum polynomials
converges coefficientwise to the minimum polynomial
$\fhat(X)$ of $\pihat_L$.  Hence $\fhat(X)$ is in
standard form.

We thus have the following degree-$p^k$ analog of a fact
which Amano proved for extensions of degree $p$ (see
Theorem~4 of \cite{amano}).

\begin{corollary} \label{main-cor}
Let $L/K$ be a totally ramified extension of degree
$p^k$ with just two distinct indices of inseparability.
Write $i_0=p^kA_0-b_0$ with $1\leqslant b_0\leqslant
p^k$.  Then there exists a uniformizer $\pihat_L$ of $L$
whose minimum polynomial has the form
\[\fhat(X)=X^{p^k} +(-1)^{b_0} \omega\pi_K^{A_0}
X^{p^k -b_0} +(-1)^{p^k} \pi_K a\]
for some $\omega \in \RR\setminus \{0\}$ and $a\in\OO_K$
with $a \equiv 1\pmod{\M_K}$.
\end{corollary}

\section{Proof of the theorem: Uniqueness}
\label{uniqueness}

Let $K$ be a local field whose residue field $\Kbar$ has
characteristic $p$.  Let $k\ge2$ and let $L/K$ be a
totally ramified separable extension of degree $p^k$
which has two distinct indices of inseparability. We
have shown that there exists a uniformizer $\pi_L$ for
$L$ whose minimum polynomial $f(X)$ is in standard form
(see Definition~\ref{standard}).  Thus $f(X)$ has only
three nonzero coefficients, and the possible values of
these coefficients are strictly limited by the indices
of inseparability of $L/K$.  In this section we show
that the polynomial $f(X)$ corresponding to $L/K$ is
uniquely determined once we have chosen a uniformizer
$\pi_K$ for $K$ and sets $\Tbar_{i_0,\omega}$ of coset
representatives for
$\overline{K}/\psibar_{i_0,\omega}(\overline{K})$.  This
will complete the proof of Theorem~\ref{main}.

\subsection{Some Lemmas} 

We begin by considering a few preliminary results.  In
all cases we assume that $L/K$ is a separable totally
ramified extension of degree $p^k$ with two distinct
indices of inseparability.  Hence by
Proposition~\ref{break}(i) we have $i_k=0$ and
$i_{k-1}=\cdots=i_1=i_0$.

\begin{proposition}
\label{prop-elem-sym}
Let $1 \leqslant h\leqslant p^k$ and set
$j=\min\{v_p(h),k\}$. Let $\sigma_1,\ldots,
\sigma_{p^k}$ be the $K$-embeddings of $L$ into a
separable closure $\Omega$ of $K$
and let $e_h(X_1,...,X_{p^k})$ be the $h$th elementary symmetric
polynomial in $p^k$ variables.  For $\alpha \in L$ set
$E_h(\alpha)=e_h(\sigma_1(\alpha), \ldots,
\sigma_{p^k}(\alpha))$.  Define $g_h:\Z\rightarrow\Z$ by
$g_h(r)=\min\{v_K(E_h(\alpha)): \alpha \in
\M_L^r\}$. Then
\[ g_h(r) \geqslant \left \lceil \frac{i_j +hr}{p^k}
\right \rceil. \]
\end{proposition}

\begin{proof} This is Theorem~4.6 of \cite{elem-sym},
specialized to the case $n=p^k$.
\end{proof}

\begin{lemma}
\label{lemmaA}
Let $\alpha\in L$ satisfy
$v_K(\alpha)=\dfrac{r}{p^k}>0$.  Then
\[ v_K(N_{L/K}(1+\alpha) -1 -N_{L/K}(\alpha)) \geqslant
\frac{i_0 +r}{p^k}, \]
with equality only if $r\equiv b_0\pmod{p^k}$.
\end{lemma}

\begin{proof}
Let $\sigma_1, \sigma_2,\ldots,\sigma_{p^k}$ be the $K$-embeddings of $L$ in the separable closure $\Omega$ of $K$. Then 
\[ N_{L/K}(1+\alpha) - 1 -
N_{L/K}(\alpha)=\sum_{h=1}^{p^k-1}E_h(\alpha). \]
Hence by Proposition~\ref{prop-elem-sym} we get
\begin{align*}
v_K(N_{L/K}(1+\alpha) -1 -N_{L/K}(\alpha)) & \geqslant
\min_{1\leqslant h\leqslant p^k-1}v_K(E_h(\alpha)) \\
& \geqslant \min_{1\leqslant h\leqslant p^k-1}
\left \lceil \frac{i_j +hr}{p^k} \right \rceil \\
& = \left \lceil \frac{i_0+r}{p^k} \right \rceil \\
& \geqslant \frac{i_0 +r}{p^k}.
\end{align*}
If equality holds then
$\displaystyle\left \lceil \frac{i_0+r}{p^k}
\right \rceil = \dfrac{i_0+r}{p^k}$, and hence $r
\equiv-i_0\equiv b_0 \pmod{p^k}$. 
\end{proof}

\begin{lemma}
\label{lemmaB}
Let $\alpha \in L$ with $v_K(\alpha)=\dfrac{r}{p^k}>0$.
Then
\begin{enumerate}[(i)]
\item $\displaystyle v_K(\alpha^{p^k}+(-1)^{p^k}N_{L/K}(\alpha))
\geqslant \frac{i_0}{p^k} +r$,
\item $\displaystyle v_K(\alpha^{p^k} -N_{L/K}(\alpha))
\geqslant \frac{i_0}{p^k} +r$.
\end{enumerate}
\end{lemma}
\begin{proof}
Let $f(X)$ be the minimum polynomial of $\alpha$ over
$K$, and set $d=[L:K(\alpha)]$.  Then
\[f(X)^d=X^{p^k}+\sum_{h=1}^{p^k}
(-1)^hE_h(\alpha)X^{p^k-h}.\]
Since $f(\alpha)=0$ we get
\[ \alpha^{p^k} +(-1)^{p^k}N_{L/K}(\alpha)
=\sum_{h=1}^{p^k-1}(-1)^{h+1}E_h(\alpha)
\alpha^{p^k-h}.\]
Therefore by Lemma~\ref{prop-elem-sym} we have
\begin{align*}
 v_K(\alpha^{p^k} +(-1)^{p^k}N_{L/K}(\alpha))& \geqslant \min_{1\leqslant h\leqslant p^k-1}\left\{ v_K(E_{h}(\alpha)) +\frac{p^k -h}{p^k}r\right\} \\
&\geqslant \min_{1\leqslant h \leqslant p^k-1} \left\{  \left\lceil \frac{i_0 +hr}{p^k} \right\rceil  +\frac{p^k-h}{p^k}r \right\} \\
&\geqslant \min_{1\leqslant h \leqslant p^k-1}\left\{ \frac{i_0 +h r}{p^k} +\frac{p^k -h}{p^k}r \right\} \\
&= \frac{i_0}{p^k} +r. 
\end{align*}
This proves (i).  If $p\neq 2$ then (ii) follows
immediately.  If $p=2$ we have
\[ \alpha^{p^k} -N_{L/K}(\alpha) =\alpha^{p^k}
+(-1)^{2^k}N_{L/K}(\alpha) -2N_{L/K}(\alpha) \]
with $v_K(\alpha^{p^k} +(-1)^{2^k}N_{L/K}(\alpha))\geqslant
\dfrac{i_0}{p^k} +r$.  Using Lemma~\ref{i0bound} we get
\[ v_K (2 N_{L/K}(\alpha)) =v_K(2) +v_K(N_{L/K}(\alpha))
= e_K +r \geqslant \frac{i_0}{p^k} +r. \]
Therefore (ii) holds in all cases.
\end{proof}

\begin{lemma}
\label{lemmaC}
Let $\alpha \in L$ with $v_K(\alpha)=\dfrac{r}{p^k}>0$.
Then
\[ v_K((1+\alpha)^{p^k} - N_{L/K}(1+\alpha)) \geqslant
\frac{i_0+r}{p^k} . \]
\end{lemma}
\begin{proof}
Observe that 
\[   (1+\alpha)^{p^k} -N_{L/K}(1+\alpha)=((1+\alpha)^{p^k} -N_{L/K}(\alpha) -1) -(N_{L/K}(1+\alpha) -N_{L/K}(\alpha) -1). \]
By Lemma~\ref{lemmaA} we have
\[ v_K(N_{L/K}(1+\alpha) - N_{L/K}(\alpha) -1) \geqslant \frac{i_0+r}{p^k}. \] 
Now observe that 
\[  (1+\alpha)^{p^k} -N_{L/K}(\alpha) -1 =(\alpha^{p^k} - N_{L/K}(\alpha)) +\sum_{h=1}^{p^k-1}\binom{p^k}{h}\alpha^{h}. \]
By Lemma~\ref{lemmaB}(ii) we have
\[ v_K (\alpha^{p^k} - N_{L/K}(\alpha)) \geqslant \frac{i_0}{p^k} +r > \frac{i_0+r}{p^k}, \]
and by Lemma~\ref{i0bound} we have
\[ v_K\left( \binom{p^k}{h} \alpha^{h}\right) \geqslant
v_K(p) +\frac{hr}{p^k} \geqslant e_K + \frac{r}{p^k}
\geqslant \frac{i_0+r}{p^k} \]
for $1\leqslant h \leqslant p^k-1$.  We conclude that 
\[ v_K((\alpha+1)^{p^k} -N_{L/K}(1+\alpha)) \geqslant
\frac{i_0+r}{p^k} . \qedhere \]
\end{proof}

\subsection{Uniqueness}
With the help of the lemmas of the previous subsection,
we now show that there is a unique polynomial in
standard form corresponding to $L/K$.  This gives the
last step in the proof of Theorem~\ref{main}.

\begin{proposition}
\label{unique}
Let $L/K$ be a totally ramified separable extension of
degree $p^k$.  Then there is a unique Eisenstein
polynomial $f(X)\in\OO_K[X]$ in standard form such that
$L$ is generated over $K$ by a root of $f(X)$.
\end{proposition}

We argue by contradiction. Suppose that $\pi_L$ and
$\pihat_L$ are uniformizers for $L$ whose minimum
polynomials $f(X)$ and $\fhat(X)$ are in standard form:
\begin{align} \label{fpoly}
f(X)&=X^{p^k} +(-1)^{b_0}\omega \pi_K^{A_0} X^{p^k-b_0} +(-1)^{p^k} \pi_K a \\
\fhat(X) &=X^{p^k} +(-1)^{b_0}\omegahat \pi_K^{A_0} X^{p^k-b_0}
+(-1)^{p^k}\pi_K\ahat. \label{gpoly}
\end{align}
Note that the index of inseparability $i_0=p^kA_0-b_0$
of $L/K$ determines $A_0$ and $b_0$.  We now show that
the coefficient of $X^{p^k-b_0}$ is the same for $f(X)$
and $\fhat(X)$.

\begin{proposition}
\label{lemma4.6}
 $\omega =\omegahat$.
\end{proposition}

\begin{proof}
We have
\[v_K(\omega\pi_K^{A_0} \pi_L^{p^k-b_0})
=v_K(\omegahat\pi_K^{A_0} \pihat_L^{p^k-b_0})
=A_0+\frac{p^k-b_0}{p^k}=\frac{i_0}{p^k}+1.\]
Since $\pi_L$ and $\pihat_L$ are uniformizers for $L$,
there is $u\in\OO_L^{\times}$ such that
$\pihat_L=u\pi_L$.  It follows that
\[\pi_K\ahat=N_{L/K}(\pihat_L)=N_{L/K}(u)N_{L/K}(\pi_L)
=N_{L/K}(u)\pi_Ka.\]
Hence $N_{L/K}(u)=\ahat/a$ is a 1-unit in $\OO_K$.
Since $L/K$ is totally ramified we have
$N_{L/K}(u)\equiv u^{p^k}\pmod{\M_L}$.  Therefore $u$ is
a 1-unit in $\OO_L$.
Since $f(\pi_L)=\fhat(\pihat_L)=0$ we get
\begin{align*}
(-1)^{b_0}(\omega\pi_K^{A_0} \pi_L^{p^k-b_0}
-\omegahat\pi_K^{A_0} \pihat_L^{p^k-b_0})
\hspace{-3cm} \\
&=(\pihat_L^{p^k} + (-1)^{p^k}N_{L/K}(\pihat_L))
-(\pi_L^{p^k}+(-1)^{p^k}N_{L/K}(\pi_L)) \\
&=(u^{p^k}-1)(\pi_L^{p^k}+(-1)^{p^k}N_{L/K}(\pi_L))
+(-1)^{p^k}N_{L/K}(\pi_L)(N_{L/K}(u)-u^{p^k}).
\end{align*}
It follows from Lemmas~\ref{lemmaB}(i) and \ref{lemmaC}
that
\begin{align*}
v_K((u^{p^k} -1)(\pi_L^{p^k}+(-1)^{p^k}N_{L/K}(\pi_L)))
&> \frac{i_0}{p^k} +1 \\
v_K(N_{L/K}(\pi_L)(N_{L/K}(u)-u^{p^k}))
&\geqslant 1 +\frac{i_0+1}{p^k}.
\end{align*}
Therefore
\[v_K(\omega \pi_K^{A_0}\pi_L^{p^k-b_0}
-\omegahat \pi_K^{A_0}\pihat_L^{p^k-b_0})
> \frac{i_0}{p^k}+1
=v_K(\omega \pi_K^{A_0}\pi_L^{p^k-b_0}).\]
Since $\omega,\omegahat\in\RR$ we conclude
that $\omega =\omegahat$.
\end{proof}

We now assume that $a\neq\ahat$, and derive a
contradiction.

\begin{lemma}
\label{lemmaD}
Let $\pi_L$ and $\pihat_L$ be as above and write
$\pihat_L =
\pi_L (1+\theta)$ with $\theta\in\M_L$. Then
$v_L(\theta) \equiv b_0 \pmod{p^k}$.
\end{lemma}

\begin{proof}
Since $\fhat(\pi_L(1+\theta))=\fhat(\pihat_L)=0$ and
$f(\pi_L)=0$ we get
\begin{align*}
0&=\pi_L^{p^k}(1+\theta)^{p^k} +(-1)^{b_0}\omega
\pi_K^{A_0} \pi_L^{p^k-b_0}(1+\theta)^{p^k-b_0} +(-1)^{p^k}\pi_K \ahat \\
&= \pi_L^{p^k}
+\pi_L^{p^k}\sum_{s=1}^{p^k}\binom{p^k}{s}\theta^s
+(-1)^{b_0} \omega \pi_K^{A_0}
\pi_L^{p^k-b_0}\left(1+\sum_{t=1}^{p^k-b_0}
\binom{p^k-b_0}{t}\theta^t\right)
+(-1)^{p^k}\pi_K \ahat \\
&=\pi_L^{p^k}\sum_{s=1}^{p^k}\binom{p^k}{s}\theta^s
+(-1)^{b_0}\omega\pi_K^{A_0}\pi_L^{p^k-b_0}
\sum_{t=1}^{p^k-b_0}\binom{p^k-b_0}{t}\theta^t
+(-1)^{p^k} \pi_K(\ahat-a).
\end{align*}
It follows that
\begin{equation}
\label{eq4}
(-1)^{p^k}\pi_K(\ahat-a)+\pi_L^{p^k}\theta^{p^k}
+\pi_L^{p^k}\sum_{s=1}^{p^k-1}\binom{p^k}{s}\theta^s
=(-1)^{b_0+1}\omega\pi_K^{A_0}\pi_L^{p^k-b_0}
\sum_{t=1}^{p^k-b_0}\binom{p^k-b_0}{t}\theta^t.
\end{equation}
By considering the term $t=1$ in the sum on the right we
see that the $K$-valuation of \eqref{eq4} is
\[ A_0 +\frac{p^k-b_0}{p^k} +v_K(\theta)=1+
\frac{i_0}{p^k} +v_K(\theta). \]
We claim that for $1\leqslant s \leqslant p^k-1$, 
\begin{equation}
\label{eq6}
v_K \left(\pi_L^{p^k}\binom{p^k}{s} \theta^{s}\right) > 1 + \frac{i_0}{p^k} +v_K(\theta).
\end{equation}
By Lemma~\ref{i0bound} we get
\[ v_K \left(\pi_L^{p^k}\binom{p^k}{s} \theta^{s}\right) \geqslant 1+ e_K + sv_K(\theta) \geqslant 1+ \frac{i_0}{p^k} +sv_K(\theta). \]
Therefore it suffices to prove \eqref{eq6} with $s=1$.
Using Lemma~\ref{i0bound} again we get
\begin{align*}
v_K \left(\pi_L^{p^k}\binom{p^k}{1} \theta^{1}\right)
&=v_K(\pi_L^{p^k} p^k \theta ) \\
&= 1+ ke_K + v_K(\theta) \\
&\geqslant 1 +\frac{ki_0}{p^k} +v_K(\theta).
\end{align*}
Since we are assuming $k\geqslant 2$ this gives
\eqref{eq6}.  It follows that
\begin{equation}
\label{eq7}
v_K ((-1)^{p^k}\pi_K(\ahat-a)+\pi_L^{p^k}\theta^{p^k} )= 1+ \frac{i_0}{p^k} +v_K(\theta). 
\end{equation}
Note that $\pi_K a=N_{L/K}(\pi_L)$ and $\pi_K \ahat =
N_{L/K}(\pihat_L)$. Since $\pihat_L = \pi_L (1+\theta)$ we get
$\pi_K \ahat= \pi_K a N_{L/K}(1+\theta)$, and hence
\begin{align} \nonumber
(-1)^{p^k}\pi_K(\ahat-a)+\pi_L^{p^k} \theta^{p^k}
&=(-1)^{p^k}\pi_Ka(N_{L/K}(1+\theta)-1)
+\pi_L^{p^k}\theta^{p^k} \\
&= (\pi_L \theta)^{p^k} +(-1)^{p^k} N_{L/K}(\pi_L \theta)
\label{ek} \\
&\hspace*{2.5cm}+(-1)^{p^k}\pi_K a (N_{L/K}(1+\theta) -1
-N_{L/K}(\theta)  ). \nonumber
\end{align}
By Lemmas~\ref{lemmaB}(i) and \ref{lemmaA} we have
\begin{align} \nonumber
v_K((\pi_L \theta)^{p^k}+(-1)^{p^k}N_{L/K}(\pi_L\theta))
&\geqslant \frac{i_0}{p^k} +1+p^k v_K(\theta) \\
&>1+ \frac{i_0}{p^k} +v_K(\theta) \label{vKpif} \\
v_K(\pi_Ka(N_{L/K}(1+\theta) -1 -N_{L/K}(\theta)) )
&\geqslant 1+\frac{i_0}{p^k} +v_K(\theta). \label{vKNLK}
\end{align}
Using \eqref{eq7}, (\ref{vKpif}), and (\ref{ek}) we
deduce that equality holds in (\ref{vKNLK}).  Hence by
Lemma~\ref{lemmaA} we have $v_L(\theta) \equiv b_0
\pmod{p^k}$.
\end{proof}

\begin{lemma} \label{lemmaZ}
Let $a$, $\ahat$ be as above.  Then
\begin{enumerate}[(i)]
\item $v_K(a-\ahat)\leqslant\B$.
\item If $v_K(a-\ahat)<\B$ then
$v_K(a-\ahat)\not\equiv b_0\pmod{p^k}$.
\end{enumerate}
\end{lemma}

\begin{proof}
By Definition~\ref{standard} we can write $\ahat$ as
\begin{equation} \label{ahat}
\ahat=\begin{cases}
\dst1+\sum_{j=1}^{\lfloor\B\rfloor}\alphahat_j
\pi_K^{j}&(\B\not\in\Z) \\[5mm]
\dst1+\left(\sum_{j=1}^{\B-1}\alphahat_j
\pi_K^{j}\right)+\gammahat\pi_K^{\B}&(\B\in\Z).
\end{cases}
\end{equation}
with $\gammahat\in T_{i_0,\omega}$, $\alphahat_j\in\RR$,
and $\alphahat_j=0$ for all $j\equiv b_0\pmod{p^k}$.
Both claims follow by comparing these formulas with the
corresponding formulas for $a$ given in
Definition~\ref{standard}.
\end{proof}

\begin{lemma}
\label{lemmaE}
With $a,\ahat$, and $\theta$ as above, we have
$v_L(\theta) =p^kv_K(\theta)= v_K(a-\ahat)$.
\end{lemma}

\begin{proof}
We can rewrite \eqref{eq4} as
\begin{equation} \label{rewrite}
\pi_L^{p^k}\theta^{p^k}
+\pi_L^{p^k}\sum_{s=1}^{p^k-1}\binom{p^k}{s}\theta^s
+(-1)^{b_0}\omega
\pi_K^{A_0}\pi_L^{p^k-b_0}
\sum_{t=1}^{p^k-b_0}\binom{p^k-b_0}{t}\theta^t
= (-1)^{p^k} \pi_K(a-\ahat).
\end{equation}
Thus we have $U+V+W=(-1)^{p^k+1} \pi_K(a-\ahat)$, with
\begin{align*}
U&=\pi_L^{p^k}\theta^{p^k}, \\
V&= \pi_L^{p^k}\sum_{s=1}^{p^k-1}\binom{p^k}{s}\theta^s,
\\
W&=(-1)^{b_0}\omega\pi_K^{A_0}\pi_L^{p^k-b_0}
\sum_{t=1}^{p^k-b_0}\binom{p^k-b_0}{t}\theta^t.
\end{align*}
As in the proof of Lemma~\ref{lemmaD} we have
\[v_K(V)> v_K(W)=1+\frac{i_0}{p^k}+v_K(\theta).\]
We consider a few cases. Suppose that $v_K(U) < v_K(W)$.
Then $v_K(U+V+W)=v_K(U)$, and hence
\[1+v_K(a-\ahat)=v_K(\pi_K(a-\ahat))=v_K(U)=1+p^kv_K(\theta).\]
Therefore the lemma holds in this case. Suppose that
$v_K(U)=v_K(W)$.  Then
\[1+p^kv_K(\theta)=1+\frac{i_0}{p^k}+v_K(\theta).\]
It follows that
$p^kv_K(\theta)=\frac{i_0}{p^k-1}=\B$ and
$v_K(U+V+W)\geqslant v_K(U)$.  Hence
\[1+v_K(a-\ahat)
\geqslant v_K(U)=1+ p^k v_K(\theta) =1+\B.\]
By Lemma~\ref{lemmaZ}(i) we have
$v_K(a-\ahat) \leqslant \B$.  Therefore
$v_K(a-\ahat)=\B= v_L(\theta)$, and the lemma holds
in this case as well. Finally, if $v_K(U)>v_K(W)$ then
\[1+p^kv_K(\theta)>1+\frac{i_0}{p^k}+v_K(\theta),\]
and hence $p^kv_K(\theta)>\frac{i_0}{p^k-1}=\B$.
Furthermore, we have $v_K(U+V+W)=v_K(W)$, so we get
\[1+v_K(a-\ahat)=v_K(W)=1+\frac{i_0}{p^k}+v_K(\theta)
>1+\frac{i_0}{p^k}+\frac{\B}{p^k}=1+ \B,\]
which contradicts Lemma~\ref{lemmaZ}(i).  Therefore this
case does not occur. It follows that $v_K(a-\ahat)=p^k
v_K(\theta)=v_L(\theta)$ holds in all cases.
\end{proof}

It follows from Lemmas~\ref{lemmaD} and \ref{lemmaE}
that $v_K(a-\ahat) \equiv b_0 \pmod{p^k}$.  Hence by
Lemma~\ref{lemmaZ} we get $v_K(a-\ahat)=\B$.  In
particular, we have $\B\in\Z$.  Since $\gamma,\gammahat$
are chosen from the set $T_{i_0,\omega}$ of liftings to
$\RR$ of coset representatives of
$\Kbar/\psibar_{i_0,\omega}(\Kbar)$, it suffices
to show that the images of $\gamma,\gammahat$ in $\Kbar$
belong to the same coset in
$\Kbar/\psibar_{i_0,\omega}(\Kbar)$.

Recall that $i_0 =(p^k-1)\B$ and $\pihat_L
=\pi_L(1+\theta)$.  By Lemma~\ref{lemmaE} we have
\[p^kv_K(\theta)=v_L(\theta)=v_K(a-\ahat)=\B.\]
Therefore there is $y\in\OO_L^{\times}$ such that
$\theta=\pi_L^{\B}y$.  Hence by (\ref{rewrite}) we have
\[ \pi_L^{(\B+1)p^k}y^{p^k}
+\sum_{s=1}^{p^k-1}\binom{p^k}{s}\pi_L^{p^k +s\B}y^{s}
+(-1)^{b_0} \sum_{t=1}^{p^k-b_0}\binom{p^k-b_0}{t}\omega
\pi_K^{A_0} \pi_L^{p^k-b_0 +t\B}y^{t}
=(-1)^{p^k}\pi_K(a-\ahat). \] 
By dividing this equation by $\pi_L^{(\B+1)p^k}$ and
using the identity $p^kA_0-b_0=(p^k-1)\B$ we get
\begin{multline}
\label{eku}
 y^{p^k}
+\sum_{s=1}^{p^k-1}\binom{p^k}{s}\pi_L^{(s-p^k)\B}y^{s}
+(-1)^{b_0} \sum_{t=1}^{p^k-b_0}\binom{p^k-b_0}{t}\omega
\pi_K^{A_0}\pi_L^{(t-1)\B-p^kA_0}y^t \\
=(-1)^{p^k} \pi_K(a-\ahat)\pi_L^{-(\B+1)p^k}.
\end{multline}
By Lemma~\ref{i0bound} we have $e_K>\frac{i_0}{p^k}
=\frac{p^k-1}{p^k}\B$.  Therefore for $1\leqslant
s\leqslant p^k-1$ we get
\[ v_K \left(\binom{p^k}{s}\pi_L^{(s-p^k)\B}\right)
\geqslant e_K +\frac{s-p^k}{p^k}\B >
\frac{p^k-1}{p^k}\B+\frac{s-p^k}{p^k}\B
=\frac{(s-1)\B}{p^k}\geqslant0. \]
Furthermore, for $2\leqslant t \leqslant p^k-b_0$ we have
\[ v_K \left(  \binom{p^k-b_0}{t} \omega \pi_K^{A_0}
\pi_L^{(t-1)\B-p^kA_0} \right)  \geqslant \frac{(t-1)\B}{p^k} > 0. \] 
Since $\pi_L^{p^k} \equiv N_{L/K}(\pi_L)\equiv
\pi_K\pmod{\M_L}$, it follows from \eqref{eku} that
\begin{equation}
\label{eq*} 
y^{p^k} +(-1)^{b_0 +1}b_0 \omega y \equiv (-1)^{p^k}
\pi_K^{-\B}(a-\ahat)\pmod{\M_L}.
\end{equation}
Since $\B=i_0/(p^k-1)$ is an integer and
$i_0=p^kA_0-b_0$ we get $A_0\equiv b_0\pmod{p^k-1}$.
If $p>2$ then $A_0\equiv b_0\pmod2$, and hence
$(-1)^{b_0}=(-1)^{A_0}$, while for $p=2$,
$(-1)^{b_0}\equiv(-1)^{A_0}\pmod2$ holds in general.
Therefore by (\ref{eq*}) we get
\begin{alignat}{2} \nonumber
y^{p^k}-(-1)^{A_0}b_0 \omega y&\equiv (-1)^{p^k}
\pi_K^{-\B}(a-\ahat)&&\pmod{\M_L} \\
(-1)^{p^k}\psi_{i_0,\omega}(y)\pi_K^{\B}&\equiv
a-\ahat&&\pmod{\M_L^{\B+1}}. \label{ekuac}
\end{alignat}
By Definition~\ref{standard} and (\ref{ahat}) we have
\begin{align*}
a&=1+\left(\sum_{j=1}^{\B-1}\alpha_j
\pi_K^{j}\right)+\gamma\pi_K^{\B} \\
\ahat&=1+\left(\sum_{j=1}^{\B-1}\alphahat_j
\pi_K^{j}\right)+\gammahat\pi_K^{\B},
\end{align*}
with $\alpha_j,\alphahat_j\in\RR$ and
$\gamma,\gammahat\in T_{i_0,\omega}$.  In addition, we
have $\alpha_j = \alphahat_j$ for $1\leqslant j<\B$.  By
substituting these formulas into \eqref{ekuac} we obtain
\[  \gamma- \gammahat \equiv
(-1)^{p^k}\psi_{i_0,\omega}(y)\pmod{\M_L}.\]
Since $\gamma,\gammahat\in T_{i_0,\omega}$ and
$(-1)^{p^k}\psi_{i_0,\omega}(y)\in\psi_{i_0,\omega}(\Kbar)$
we get $\gamma=\gammahat$, and hence $a=\ahat$. 

This completes the proof of the
Proposition~\ref{unique}, and hence also completes the
proof of Theorem~\ref{main}.

\section{Splitting fields}
\label{splitting}

In this section we compute the splitting field in
$\Omega$ of a polynomial
\begin{equation}
\label{poly}
f(X)= X^{p^k} +(-1)^{b_0}\omega \pi_K^{A_0} X^{p^k -b_0} +(-1)^{p^k}\pi_K a 
\end{equation}
in standard form over $\OO_K$.  This allows us to
determine when the extension $L/K$ generated by a root
of $f(X)$ is Galois over $K$.

We first need to prove the following basic result.

\begin{proposition} \label{tamesplit}
Let $K$ be a local field with residue characteristic $p$
and let $n\geqslant1$ with $p\nmid n$.  Let
\begin{align*}
f(X)=X^n+\sum_{i=1}^nc_iX^{n-i}, \hspace{1cm}
g(X)=X^n+\sum_{i=1}^nd_iX^{n-i}
\end{align*}
be monic polynomials with coefficients in $K$ such that
\begin{enumerate}[(i)]
\item $v_K(c_n)=v_K(d_n)=r$ for some $r\in\Z$,
\item $c_n\equiv d_n\pmod{\M_K^{r+1}}$,
\item $v_K(c_i)>ir/n$ and $v_K(d_i)>ir/n$ for
$1\leqslant i\leqslant n-1$.
\end{enumerate}
Then the splitting field of $f(X)$ over $K$ is the same
as the splitting field of $g(X)$ over $K$.
\end{proposition}

\begin{proof}
Let $L_f$, $L_g$ denote the splitting fields of $f(X)$,
$g(X)$ over $K$.  Let $\beta\in L_g$ be a root of
$g(X)$.  Then by the assumptions on $v_K(d_i)$ we have
$v_K(\beta)=r/n$.  Hence
\[f(\beta)=f(\beta)-g(\beta)
=\sum_{i=1}^n(c_i-d_i)\beta^{n-i}\]
satisfies $v_K(f(\beta))>r$.  Set
$\ftilde(X)=\beta^{-n}f(\beta X)$.  Then
$\ftilde(X)\in K(\beta)[X]$ is monic, and for
$1\leqslant i\leqslant n-1$ the coefficient
$\ctilde_i=\beta^{-i}c_i$ of $X^{n-i}$ in $\ftilde(X)$
has positive valuation.  Furthermore, we have
$v_K(\ctilde_n)=0$ and $v_K(\ftilde(1))>0$.  Since
$p\nmid n$ it follows by Hensel's lemma that there is
$u\in L_g$ with $u\equiv1\pmod{\M_{L_g}}$ and
$\ftilde(u)=0$.  Set $\alpha=u\beta$.  Then
$\alpha\in L_g$, and since $f(\alpha)=0$ we also have
$\alpha\in L_f$.  Set $\fhat(X)=\alpha^{-n}f(\alpha X)$.
Then $L_f$ is the splitting field of $\fhat(X)$ over
$K(\alpha)$.
As above we see that for $1\leqslant i\leqslant n-1$
the coefficient $\chat_i=\alpha^{-i}c_i$ of $X^{n-i}$ in
$\fhat(X)$ has positive valuation.  Since $\fhat(X)$ is
monic and $\fhat(1)=0$, the constant term $\chat_n$ of
$\fhat(X)$ satisfies $\chat_n\equiv-1\pmod{\M_{L_f}}$.
Since $p\nmid n$ it follows by Hensel's lemma that
$L_f=K(\alpha,\zeta_n)$, where $\zeta_n\in\Omega$ is a
primitive $n$th root of unity.  By symmetry we get
$\zeta_n\in L_g$, and hence $L_f\subset L_g$.  Using
symmetry once again we get $L_g\subset L_f$, and hence
$L_f=L_g$.
\end{proof}

\begin{corollary} \label{kummer}
Let $K$ be a local field with residue characteristic $p$
and let $n\geqslant1$ with $p\nmid n$.  Let
\[f(X)=X^n+\sum_{i=1}^nc_iX^{n-i}\in K[X],\]
and assume there is $r\in\Z$ such that $v_K(c_n)=r$ and
$v_K(c_j)>jr/n$ for $1\leqslant j\leqslant n-1$.  Let
$a\in K$ with $a\equiv c_n\pmod{\M_K^{r+1}}$.  Then the
splitting field of $f(X)$ over $K$ is
$K(\zeta_n,\sqrt[n]{-a})$.
\end{corollary}

\begin{proof}
Apply the proposition with $g(X)=X^n+a$.
\end{proof}

\begin{proposition}
The splitting field over $K$ of the polynomial $f(X)$
from (\ref{poly}) is
\[ 
K\left(\pi_L, \zeta_{p^k-1},
\sqrt[p^k-1]{(-1)^{b_0}b_0\omega \pi_K^{p^k-b_0+A_0-1}}\right),
 \]
where $\zeta_{p^k-1}$ is a primitive $p^k-1$ root of
unity in the separable closure $\Omega$ of $K$. 
\end{proposition}

\begin{proof}
Let $M$ be the splitting field of $f(X)$ over $K$, let
$\pi_L\in M$ be a root of $f(X)$, and set $L=K(\pi_L)$.
Let $g(Y)=f(\pi_L+\pi_LY)\in L[X]$.  Then the splitting
field of $g(Y)$ over $L$ is $M$.  Since $f(\pi_L)=0$ we
get
\begin{align*}
g(Y)&=\pi_L^{p^k}(1+Y)^{p^k} +(-1)^{b_0}
\omega\pi_K^{A_0} \pi_L^{p^k-b_0}
(1+Y)^{p^k-b_0} +(-1)^{p^k}\pi_K a \\
&=\pi_L^{p^k}\sum_{s=1}^{p^k}\binom{p^k}{s}
Y^s+(-1)^{b_0}\omega\pi_K^{A_0} \pi_L^{p^k-b_0}
\sum_{t=1}^{p^k-b_0}\binom{p^k-b_0}{t}Y^t.
\end{align*}
Set $h(Y)=\pi_L^{-p^k}Y^{-1}g(Y)$.  Then
\begin{align*}
h(Y)&=\sum_{s=0}^{p^k-1}\binom{p^k}{s+1}
Y^s+(-1)^{b_0}\omega\pi_K^{A_0} \pi_L^{-b_0}
\sum_{t=0}^{p^k-b_0-1}\binom{p^k-b_0}{t+1}Y^t
\end{align*}
is a monic polynomial over $L$ of degree $p^k-1$ whose
splitting field is $M$.  Therefore we can write
\[h(Y)=Y^{p^k-1}+\sum_{j=1}^{p^k-1}a_jY^{p^k-1-j}\]
with $a_j\in L$.  The constant term of $h(Y)$ is
\begin{align*}
a_{p^k-1}
&=p^k+(p^k-b_0)(-1)^{b_0}\omega\pi_K^{A_0}\pi_L^{-b_0}.
\end{align*}
Hence by Lemma~\ref{i0bound} we get
$v_L(a_{p^k-1})=p^kA_0-b_0=i_0$ and
\begin{equation} \label{apk1cong}
a_{p^k-1}\equiv-b_0(-1)^{b_0}\omega\pi_K^{A_0}\pi_L^{-b_0}
\pmod{\M_L^{p^kA_0-b_0+1}}.
\end{equation}
On the other hand, for $0\leqslant s\leqslant p^k-2$ we
have
\begin{align*}
v_L\left(\binom{p^k}{s+1}\right)
&\geqslant p^ke_K>i_0,
\end{align*}
again by Lemma~\ref{i0bound}.  For
$0\leqslant t\leqslant p^k-b_0-1$ we have
\begin{align*}
v_L\left((-1)^{b_0}\omega\pi_K^{A_0} \pi_L^{-b_0}
\binom{p^k-b_0}{t+1}\right)
&\geqslant p^kA_0-b_0=i_0.
\end{align*}
Hence $v_L(a_j)\geqslant i_0$ for
$0\leqslant j\leqslant p^k-2$.

Now define $\hhat(Y)=\pi_L^{(p^k-1)(p^k-b_0)}
h(\pi_L^{b_0-p^k}Y)$ and write
\[\hhat(Y)=Y^{p^k-1}+\sum_{j=1}^{p^k-1}\ahat_jX^{p^k-1-j}.\]
Using (\ref{apk1cong}) and the congruence
$\pi_L^{p^k}\equiv\pi_K\pmod{\M_L^{p^k+1}}$ we get
\begin{alignat*}{2}
\ahat_{p^k-1}&=\pi_L^{(p^k-1)(p^k-b_0)}a_{p^k-1} \\
&\equiv-b_0(-1)^{b_0}\omega\pi_K^{A_0}
\pi_L^{p^{2k}-p^k-p^kb_0}
&&\pmod{\M_L^{p^k(A_0+p^k-1-b_0)+1}} \\
&\equiv-b_0(-1)^{b_0}\omega\pi_K^{A_0+p^k-1-b_0}
&&\pmod{\M_L^{p^k(A_0+p^k-1-b_0)+1}}.
\end{alignat*}
For $1\leqslant j\leqslant p^k-2$ we have
\[\frac{v_L(a_j)}{j}\geqslant\frac{i_0}{j}
>\frac{i_0}{p^k-1}=\frac{v_L(a_{p^k-1})}{p^k-1}.\]
It follows that $v_L(\ahat_j)/j>
v_L(\ahat_{p^k-1})/(p^k-1)$, so $\hhat(Y)$ satisfies the
hypotheses of Corollary~\ref{kummer}.  Since $M$ is the
splitting field of $\hhat(Y)$ over $L$ we get
\begin{align*}
M&=L\left(\zeta_{p^k-1},\sqrt[p^k-1]{(-1)^{b_0}
b_0\omega\pi_K^{p^k-b_0+A_0-1}}\right) \\
&=K\left(\pi_L,\zeta_{p^k-1},\sqrt[p^k-1]{(-1)^{b_0}
b_0\omega\pi_K^{p^k-b_0+A_0-1}}\right). \qedhere
\end{align*}
\end{proof}

\begin{corollary} \label{Gal}
Let $L$ be generated over $K$ by a root of the
polynomial (\ref{poly}).  Then $L/K$ is Galois if and
only if $\zeta_{p^k-1}\in K$ and there is a $p^k-1$
root of $(-1)^{b_0}b_0\omega\pi_K^{p^k-b_0+A_0-1}$ in
$K$.  In this case $\Gal(L/K)$ is an elementary abelian
$p$-group of rank $k$.
\end{corollary}

\begin{proof}
Let $F$ be the splitting field of
$g(X)=X^{p^k-1}-(-1)^{b_0}b_0\omega\pi_K^{p^k-b_0+A_0-1}$
over $K$.  It follows from the proposition that $L/K$ is
Galois if and only if $F\subset L$.  Since $F/K$ is at
most tamely ramified, and $L/K$ is totally wildly
ramified, this holds if and only if $F=K$.  This proves
the first claim.  Suppose $L/K$ is Galois.  By
Proposition~\ref{break}(ii), $\B=\frac{i_0}{p^k-1}$ is the only
ramification break of $L/K$.  Hence $\Gal(L/K)$ is an
elementary abelian $p$-group.
\end{proof}

\section{Some examples}
\label{examples}

In this section we give an example which illustrates
Theorem~\ref{main}.  We also give examples which rule
out some plausible approaches to generalizing this
theorem.

\subsection{Classifying extensions with specified
indices of inseparability}

Let $K=\Q_3(\zeta_8)$ be the unramified extension of
$\Q_3$ of degree 2.  In this example we determine all
totally ramified extensions $L/K$ of degree $3^2=9$ such
that $i_0=i_1=8$ and $i_2=0$.  Since $8=3^2\cdot1-1$ we
get $A_0=b_0=1$ in this case.  It follows from
Theorem~\ref{main} that the isomorphism classes of such
extensions are generated by roots of polynomials of the
form $X^9-3\omega X^8-3a$, where
$\omega\in\RR\smallsetminus\{0\}
=\langle\zeta_8\rangle$ and $a=1+3\gamma$ with
$\gamma\in T_{8,\omega}$.  We have
$\psi_{8,\omega}=X^9+\omega X$.  Since $\Kbar\cong\F_9$
it follows that $\psibar_{8,-1}(\alpha)=0$ for all
$\alpha\in\Kbar$.  Hence $\Tbar_{8,-1}=\Kbar$, so
$T_{8,-1}=\RR$.  Now suppose
$\omega\in\langle\zeta_8\rangle$ with $\omega\not=-1$.
Then the only root in $\Kbar$ of
$\psibar_{8,\omega}$ is 0, so
$\psibar_{8,\omega}$ is onto.  Therefore we can
take $\Tbar_{8,\omega}=\{0\}$, and hence
$T_{8,\omega}=\{0\}$.  It follows that the totally
ramified extensions of $K$ of degree 9 such that
$i_0=i_1=8$ are precisely
those generated by a root of one of the 16 polynomials
\begin{alignat}{2} \label{minusone}
&X^9+3X^8-3(1+3\gamma)\;\;\;&&(\gamma\in\RR), \\
&X^9-3\omega X^8-3&&(\omega\in\langle\zeta_8\rangle,\;
\omega\not=-1). \label{omega}
\end{alignat}

We can ask which of these extensions are Galois.  It
follows from Corollary~\ref{Gal} that we get a
Galois extension if and only if
$(-1)^{1+1}(9-1)\omega\cdot3^{9-1+1-1}=8\omega\cdot3^8$
is an 8th power in $K$.  By Hensel's lemma this is
equivalent to $-\overline{\omega}$ being an 8th power in
$\Kbar$, which holds if and only if $\omega=-1$.  Hence
the 9 Eisenstein polynomials in (\ref{minusone}) give
Galois extensions, while the 7 polynomials in
(\ref{omega}) do not.  If $L/K$ is the Galois extension
associated to a polynomial from (\ref{minusone}) then
$\Gal(L/K)\cong(\Z/3\Z)^2$ by Corollary~\ref{Gal}.

\subsection{Extensions with $d$ indices of inseparability} 

Let $L/K$ be a separable totally ramified extension of
local fields of degree $p^k$.  If $L/K$ has just two
distinct indices of inseparability then $L$ is generated
over $K$ by a root of an Eisenstein polynomial with
three terms.  One might hope that in the cases where
$L/K$ has $d\geqslant 3$ distinct indices of
inseparability then $L$ can be generated over $K$ by a
root of an Eisenstein polynomial with $d+1$ terms.  This
does not hold in general, as we now demonstrate.

Let $K$ be a local field with residue characteristic 3
and let $\pi_K$ be a uniformizer for $K$.  Let $L/K$ be
an extension obtained by adjoining to $K$ a root $\pi_L$
of the Eisenstein polynomial
\begin{align*}
f(X)&=X^9 + \pi_K X^7 -\pi_K X^6 +\pi_K X^{3} -\pi_K \\
&=X^9+\sum_{i=1}^9(-1)^ic_iX^{9-i}.
\end{align*}
Then $L/K$ is a totally ramified extension of degree $9$
which has three distinct indices of inseparability:
$i_0=7$, $i_1=3$, and $i_2=0$. We show that $L$ is not
generated over $K$ by a root of an Eisenstein polynomial
with fewer than five nonzero terms.

Let $\pitilde_L$ be a uniformizer for $L$.  Then there
are $r_0,r_1\in\RR$ with $r_0\not=0$ such that
\begin{equation} \label{piLcong}
\pitilde_L\equiv r_0\pi_L+r_1r_0^2\pi_L^2
\pmod{\M_L^3}.
\end{equation}
Let $\pihat_L=r_0\pi_L$.  Then the minimum polynomial
for $\pihat_L$ over $K$ is
\begin{align*}
\fhat(X)
&=X^9+r_0^2\pi_K X^7-r_0^3\pi_K X^6
+r_0^6\pi_K X^{3} -r_0^9\pi_K \\
&=X^9+\sum_{i=1}^9(-1)^i\chat_iX^{9-i}.
\end{align*}
Furthermore, we have
\begin{equation} \label{piLtcong}
\pitilde_L\equiv\pihat_L+r_1\pihat_L^2\pmod{\M_L^3}.
\end{equation}
We use Theorem~\ref{perturbation} with $\ell=1$ to
obtain congruences for the coefficients of the minimum
polynomial
\[\ftilde(X)=X^9+\sum_{i=1}^9(-1)^i\ctilde_iX^{9-i}\]
for $\pitilde_L$ over $K$.  The values of $j$ that
satisfy the hypothesis of the theorem are $j=0$ and
$j=1$.  Taking $j=0$ we get $\phi_0(1)=8$, $h=h_0=1$,
$t=1$, $S_0=\{0\}$, and $A_0=b_0=1$.  Therefore
\begin{alignat*}{2}
\ctilde_1&\equiv\chat_1+(-1)^{1+1+1}(1+1)r_1\chat_2
&&\pmod{\M_K^2} \\
&\equiv r_1r_0^2\pi_K&&\pmod{\M_K^2}.
\end{alignat*}
When $j=1$ we have $\phi_1(1)=6$, $h=3$, $h_0=1$, $t=1$,
$S_1=\{1\}$, $A_1=1$, and $b_1=3$.  Therefore we get
\begin{alignat*}{2}
\ctilde_3&\equiv\chat_3+(-1)^{1+1+1}(1+1)r_1^3\chat_6
&&\pmod{\M_K^2} \\
&\equiv r_0^3\pi_K + r_1^3r_0^6\pi_K&&\pmod{\M_K^2}. 
\end{alignat*}
Hence if $r_1\not=0$ then $v_K(\ctilde_1)=1$, while if
$r_1=0$ then $v_K(\ctilde_3)=1$.  In particular, at
least one of $\ctilde_1,\ctilde_3$ is nonzero.  Since
$i_0=7$ we have $\ctilde_2\not=0$, and since $i_1=3$ we
have $\ctilde_6\not=0$.  Together with $X^9$ and
$\ctilde_9$ this gives at least five nonzero terms in
$\ftilde(X)$, as claimed.

\subsection{Extensions of degree $n\neq p^k$} 
One might hope to generalize Theorem~\ref{main}, or at
least Corollary~\ref{main-cor}, to include totally
ramified extensions $L/K$ whose degree $n=[L:K]$ is not
a power of $p$.  The following example shows that such a
generalization is not possible, even in the case
$v_p(n)=1$.

Let $K$ be a local field with residue characteristic 3
and let $\pi_L$ be a uniformizer for $K$.  Let $L/K$ be a
totally ramified extension of degree $6$ obtained by
adjoining to $K$ a root of the Eisenstein polynomial
\begin{align} \label{deg6}
f(X)&=X^{6} -\pi_K X^5 +\pi_K X^{4}   + \pi_K \\
&=X^6+\sum_{i=1}^6(-1)^ic_iX^{6-i}. \nonumber
\end{align}
Then $L/K$ is totally ramified of degree 6 and has
indices of inseparability $i_1=0$ and $i_0=4$. We show
that there does not exist a uniformizer for $L$ whose
minimum polynomial over $K$ has fewer than four nonzero
terms.

Let  $\pitilde_L$ be a uniformizer for $L$ whose minimum polynomial is given by 
\[\ftilde(X)=X^6+\sum_{i=1}^6
(-1)^i\ctilde_iX^{6-i}.\]
As in the previous
example there are $r_0,r_1\in\RR$ such that
$\pitilde_L$ satisfies (\ref{piLcong}).  We find
that the minimum polynomial for
$\pihat_L=r_0\pi_L$ over $K$ is
\begin{align*}
\fhat(X)
&=X^6-r_0\pi_K X^5+r_0^2\pi_K X^4+r_0^6\pi_K \\
&=X^6+\sum_{i=1}^6(-1)^i\chat_iX^{6-i},
\end{align*}
and the congruence (\ref{piLtcong}) from the previous
example holds.  We apply Theorem~\ref{perturbation} with
$\ell=1$ to obtain congruences for the coefficients of
the minimum polynomial
\[\ftilde(X)=X^6+\sum_{i=1}^6
(-1)^i\ctilde_iX^{6-i}\]
for $\pitilde_L$ over $K$.  Once again, the values of
$j$ that satisfy the hypothesis of the theorem are
$j=0$ and $j=1$.  Taking $j=0$ gives $\phi_0(1)=5$,
$h=h_0=1$, $t=1$, $S_0=\{0\}$, $A_0=1$, and $b_0=2$.
Therefore we get
\begin{alignat*}{2}
\ctilde_1&\equiv\chat_1+(-1)^{1+1+1}(1+1)r_1\chat_2
&&\pmod{\M_K^2} \\
&\equiv r_0\pi_K+r_1r_0^2\pi_K&&\pmod{\M_K^2}.
\end{alignat*}
When $j=1$ we have $\phi_1(1)=3$, $h=3$, $h_0=1$, $t=1$,
$S_1=\{1\}$, $A_1=1$, and $b_1=6$. Therefore we get
\begin{alignat*}{2}
\ctilde_3&\equiv\chat_3+(-1)^{1+1+1}\cdot2r_1^3\chat_6
&&\pmod{\M_K^2} \\
&\equiv r_1^3r_0^6\pi_K&&\pmod{\M_K^2}. 
\end{alignat*}
Hence if $r_1=0$ then $v_K(\ctilde_1)=1$, while if
$r_1\not=0$ then $v_K(\ctilde_3)=1$.  In particular, at
least one of $\ctilde_1$, $\ctilde_3$ is nonzero.  Since
$i_0=4$ we have $\ctilde_2\not=0$.  Together with $X^6$
and $\ctilde_6$ this gives at least four nonzero terms in
$\ftilde(X)$, as claimed.

\subsection{Indices of inseparability and the number of
terms in an Eisenstein polynomial}

Here we give an example that shows that there are
totally ramified extensions $L/K$ and $M/K$ of degree
$p^k$ with the same indices of inseparability such that
$M$ is generated over $K$ by a root of an Eisenstein
polynomial with three nonzero terms, but $L$ is not.
This implies that the indices of inseparability by
themselves are not enough to characterize totally
ramified extensions of degree $p^k$ which are generated
by a root of a 3-term Eisenstein polynomial.

Let $p$ be an odd prime and let $K$ be a finite
extension of $\mathbb{Q}_p$. Let
\begin{align*}
f(X) &= X^{p^2} +(-1)^{b_1-1}\pi_K^{A_1+e_K} X^{p^2 -b_1 +1}
+(-1)^{b_1}\pi_K^{A_1} X^{p^2-b_1} -\pi_K \\
g(X) &=X^{p^2} +(-1)^{b_1}\pi_K^{A_1} X^{p^2 -b_1} -\pi_K
\end{align*}
with $p \mid b_1$, $p<b_1 <p^2$, and $1\leqslant A_1 \leqslant
e_K$. Let $L$ be the totally ramified extension of $K$
obtained by adjoining a root of $f(X)$ to $K$ and let
$M/K$ be the totally ramified extension obtained by
adjoining a root of $g(X)$ to $K$. Then $L/K$ and $M/K$
have the same indices of inseparability:
\[ i_0 = p^2(A_1 +e_K)-b_1,\; i_1= p^2 A_1 -b_1,\; i_2=0.\]
We claim that $L$ is not generated over $K$ by a root of
a three-term Eisenstein polynomial.

Suppose there exists a uniformizer $\pitilde_L$
for $L$ whose minimum polynomial
\[\ftilde(X)=X^{p^2}+\sum_{i=1}^{p^2}
(-1)^i\ctilde_iX^{p^2-i}\]
has just three nonzero terms.  Since $i_1=p^2A_1-b_1$ is
an index of inseparability of $L/K$, $\ftilde(X)$ must
have the form
\begin{equation} \label{f3}
\ftilde(X)=X^{p^2}
+(-1)^{b_1}\ctilde_{b_1}X^{p^2-b_1}
-\ctilde_{p^2}
\end{equation}
for some $\ctilde_{b_1},\ctilde_{p^2}\in K$
with $v_K(\ctilde_{b_1})=A_1$ and
$v_K(\ctilde_{p^2})=1$.  Let $\pi_L$ be a root of $f(X)$
which generates $L$ over $K$.  As in the previous
examples there are $r_0,r_1\in\RR$ such that
$\pitilde_L$ satisfies (\ref{piLcong}).  We find that
the minimum polynomial for $\pihat_L=r_0\pi_L$ over $K$
is
\begin{align*}
\fhat(X)
&=X^{p^2}+(-1)^{b_1-1}r_0^{b_1-1}\pi_K^{A_1+e_K}X^{p^2-b_1+1}
+(-1)^{b_1}r_0^{b_1}\pi_K^{A_1}X^{p^2-b_1}-r_0^{p^2}\pi_K \\
&=X^{p^2}+\sum_{i=1}^{p^2}(-1)^i\chat_iX^{{p^2}-i},
\end{align*}
and the congruence (\ref{piLtcong}) holds.  We apply
Theorem~\ref{perturbation} with $\ell=1$ to obtain
congruences for the coefficients of $\ftilde(X)$.  As in
the previous two examples, the values of $j$ that
satisfy the hypothesis of the theorem are $j=0$ and
$j=1$.  When $j=0$ we have $\phi_0(1)=p^2(A_1+e_K)-b_1+1$,
$h=h_0=b_1-1$, $t=A_1+e_K$, $S_0=\{0\}$, $A_0=A_1+e_K$,
and $b_0=b_1$. Therefore we get
\begin{alignat*}{2}
\ctilde_{b_1-1}&\equiv\chat_{b_1-1}
+(-1)^{A_1+e_K+1+A_1+e_K}(b_1-1+1)\chat_{b_0}r_1
&&\pmod{\M_K^{A_1+e_K+1}} \\
&\equiv r_0^{b_1-1}\pi_K^{A_1+e_K}-b_1
\pi_K^{A_1}r_0^{b_1}r_1
&&\pmod{\M_K^{A_1+e_K+1}}.
\end{alignat*}
When $j=1$ we have $\phi_1(1)=p^2A_1-b_1+p$, $h=b_1-p$,
$h_0=p^{-1}b_1-1$, $t=A_1$, and $S_1=\{1\}$.  Therefore
we get
\begin{alignat*}{2}
\ctilde_{b_1-p}&\equiv\chat_{b_1-p}
+(-1)^{A_1+1+A_1}(p^{-1}b_1-1+1)\chat_{b_1}r_1^p
&&\pmod{\M_K^{A_1+1}} \\
&\equiv-p^{-1}b_1\pi_K^{A_1}r_0^{b_1}r_1^p
&&\pmod{\M_K^{A_1+1}}. 
\end{alignat*}
Hence if $r_1=0$ then
$v_K(\ctilde_{b_1-1})=A_1+e_K$, while if
$r_1\not=0$ then $v_K(\ctilde_{b_1-p})=A_1$.  In
particular, at least one of
$\ctilde_{b_1-1},\ctilde_{b_1-p}$ is
nonzero.  This contradicts the fact that
$\ftilde(X)$ has the form given in (\ref{f3}).
Hence $L$ is not generated over $K$ by a root of an
Eisenstein polynomial with three terms.

\bibliographystyle{plain}
\bibliography{ThreeEisenstein}

\end{document}